\documentclass[oneside,a4paper,leqno,11pt]{amsart}
\usepackage{latexsym,amsfonts,amssymb,amsmath,amscd}
\usepackage{graphicx,color,ifpdf}
\allowdisplaybreaks

\usepackage[bookmarks=true,pdfstartview=FitH, pdfborder={0 0 0},
colorlinks=true,citecolor=blue, linkcolor=blue]{hyperref}

\usepackage{aliascnt}

\newtheorem{theorem}{Theorem}[section]
\newaliascnt{lemma}{theorem}
\newtheorem{lemma}[lemma]{Lemma}
\aliascntresetthe{lemma}

\newaliascnt{proposition}{theorem}
\newtheorem{proposition}[proposition]{Proposition}
\aliascntresetthe{proposition}

\newaliascnt{corollary}{theorem}

\aliascntresetthe{corollary}

\newaliascnt{claim}{theorem}
\newtheorem{claim}[claim]{Claim}
\aliascntresetthe{claim}

\theoremstyle{definition}
\newaliascnt{definition}{theorem}
\newtheorem{definition}[definition]{Definition}
\aliascntresetthe{definition}

\newtheorem*{acknowledgements}{Acknowledgements}

\newaliascnt{example}{theorem}

\aliascntresetthe{example}

\newaliascnt{exercise}{theorem}

\aliascntresetthe{exercise}

\newaliascnt{question}{theorem}

\aliascntresetthe{question}

\newaliascnt{problem}{theorem}

\aliascntresetthe{problem}

\theoremstyle{remark}
\newaliascnt{remark}{theorem}
\newtheorem{remark}[remark]{Remark}
\aliascntresetthe{remark}

\newaliascnt{notation}{theorem}

\aliascntresetthe{notation}

\newaliascnt{fact}{theorem}

\aliascntresetthe{fact}

\numberwithin{equation}{theorem}%
\numberwithin{figure}{theorem}

\renewcommand{\Bbb}[1]{\mathbb{#1}}
\newcommand{\abs}[1]{\left|#1\right|}

\newcommand{\diam}{\operatorname{diam}}
\newcommand{\dist}{\operatorname{dist}}

\begin{document}
\title[almost nilpotency and Hurewicz fibration]{Margulis lemma and Hurewicz fibration Theorem on Alexandrov spaces}
\author{Shicheng Xu}
\email{shichengxu@gmail.com}
\address{School of Mathematical Sciences, Capital Normal University, Beijing, China}

\author{Xuchao Yao}
\email{qiujuli@vip.sina.com}
\address{School of Mathematical Sciences, Capital Normal University, Beijing, China}


\date{\today}

\begin{abstract}
We prove the generalized Margulis lemma with a uniform index bound on an Alexandrov $n$-space $X$ with curvature bounded below, i.e., small loops at $p\in X$ generate a subgroup of the fundamental group of unit ball $B_1(p)$ that contains a nilpotent subgroup of index $\le w(n)$, where $w(n)$ is a constant depending only on the dimension $n$. The proof is based on the main ideas of V.~Kapovitch, A.~Petrunin, and W.~Tuschmann, and the following results:

(1) We prove that any regular almost Lipschitz submersion constructed by Yamaguchi on a collapsed Alexandrov space with curvature bounded below is a Hurewicz fibration. We also prove that such fibration is uniquely determined up to a homotopy equivalence.

(2) We give a detailed proof on the gradient push, improving the universal pushing time bound given by V.~Kapovitch, A.~Petrunin, and W.~Tuschmann, and justifying in a specific way that the gradient push between regular points can always keep away from extremal subsets.
\end{abstract}
\maketitle

\section{Introduction}
In this paper we prove the Margulis lemma on Alexandrov spaces with curvature bounded below. A group $\Gamma$ is called {\it $w$-nilpotent} if there is a nilpotent subgroup $N<\Gamma$ whose index $[\Gamma:N]\le w$.
Let $B_r(p)$ denote a metric ball centered at $p$ of radius $r$.
\begin{theorem}[Generalized Margulis Lemma]\label{thmb-margulis}
	There are $\epsilon(n),w(n)>0$ such that for any Alexandrov space $X$ with curvature $\ge -1$ and any point $p\in X$, the subgroup $\Gamma_p(p;\epsilon)$ of fundamental group $\pi_1(B_1(p),p)$ generated by loops at $p$ lying in $B_{\epsilon}(p)$ with $0<\epsilon\le \epsilon(n)$ is $w(n)$-nilpotent.
\end{theorem} 

The original Margulis lemma is also called Margulis-Heintze's theorem, which was proved by Margulis (cf. \cite{Gr78a}), and also independently discovered by Heintze \cite{He76} on manifolds of $-1\le K\le 0$. Since then, it has been one of the fundamental facts in Riemannian geometry which has many applications, e.g., Gromov's almost flat theorem \cite{Gr78a}, finiteness of closed negatively pinched manifolds \cite{Gr78b} of bounded volume, and more recently the almost rigidity of maximal volume entropy \cite{CRX16} for manifolds of lower bounded Ricci curvature to be hyperbolic. 

For manifolds with sectional curvature $K\ge -1$, it was proved by Fukaya-Yamaguchi \cite{FY92} that $\Gamma(\epsilon)$ is almost nilpotent without a uniform bound on the index. For Alexandrov spaces, the earlier version of Theorem \ref{thmb-margulis} was proved by Yamaguchi in \cite{Ya96}, also without a uniform bound on the index of the nilpotent subgroup, where the proof was based on the
Lipschitz submersion Theorem \ref{thm:Lipschitz-submersion}
and arguments in \cite{FY92}. Later a global version of Theorem \ref{thmb-margulis} for manifolds of almost nonnegative curvature was proved by Kapovitch-Petrunin-Tuschmann \cite{KPT10}, where $\Gamma(\epsilon)$ admits a nilpotent subgroup with uniformly bounded index. Theorem \ref{thmb-margulis} also follows from the main ideas of Kapovitch-Petrunin-Tuschmann \cite{KPT10}. 

Gromov conjectured that the Margulis lemma with a universal bounded index holds for manifolds of lower bounded Ricci curvature. A breakthrough on this conjecture was made by Cheeger-Colding \cite{CC96}, and it has been finally confirmed recently by
Kapovitch-Wilking \cite{KW11}. 

We point it out that the uniform index bound is very important to some geometric applications, for example, in Gromov's almost flat theorem \cite{Gr78a}, the uniform index bound corresponds to the holonomy gap which is crucial in Gromov's and Ruh's proof (see \cite{Gr78a}, \cite{Ruh82}, \cite{BK81}). The uniformly index bound is also crucial for the almost rigidity of maximal volume entropy \cite{CRX16} in deriving that the connectedness component of a Gromov-Hausdorff limit group of deck-transformations is a nilpotent Lie group.

\begin{remark}
	More generally, one may further consider a metric space $X$ of $K$-bounded packing, i.e., there is $K>0$ such that every ball of radius $4$ in $X$ can be covered by at most $K$ balls of radius $1$. In \cite[\S5.F]{Gro2007} Gromov proposed a question whether a discrete isometric subgroup $\Gamma$ acting on a metric space with $K$-bounded packing is virtually nilpotent, if $\Gamma$ is generated by finite elements whose displacement at one point $< \epsilon(K)$?  It has been answered affirmatively by \cite{BGT2012} recently. However, the uniform index bound as in Theorem \ref{thmb-margulis} is beyond their approach (see \cite[Section 11]{BGT2012}). 
\end{remark}

Our proof relies on Theorem \ref{thm:Lipschitz-submersion} and Theorem \ref{thm-gradientpush-1} below.

For small $0<\delta<\delta(n,\kappa)$, the {\it $\delta$-strained radius} \cite{Ya96} at a
point $p$ in an $n$-dimensional Alexandrov space $Y$ of curv $\ge \kappa$ is
defined to be $$r_{\text{$\delta$-str}}(p)=\sup \{r\; | \text{ there exists an
	$(n,\delta)$-strainer
	at $p$ of length $r$} \}.$$ 
Let $r_{\text{$\delta$-str}}(Y)=\inf\{  	r_{\text{$\delta$-str}}(p):p\in Y\}$.
Let $\varkappa(\delta,\epsilon|n)$ denote a positive
function depending on $n$, $\delta$ and $\epsilon$ satisfying
$\varkappa(\delta,\epsilon)\to 0$ as $\delta,\epsilon\to 0$.
A map $f:X\to Y$ between Alexandrov spaces is an
{\it $\epsilon$-almost Lipschitz submersion} \cite{Ya96} if
\begin{enumerate}
	\item[(i)] $f$ is an $\epsilon$-Gromov-Hausdorff approximation (GHA for simplicity), i.e., for any $p,q\in X$, $||f(p)f(q)|-|pq||\le \epsilon$ and $f(X)$ is $\epsilon$-dense in $Y$, where $|pq|=d(p,q)$ denote the distance between two points $p,q$; and
	\item[(ii)] for any $p, q \in X$, 
	$$\left|\frac{|f(p)f(q)|}{|pq|}-\sin\theta\right|<\epsilon,$$ where $\theta(p,q)$ is the infimum of
	$\measuredangle qpx$ when $x$ runs over $f^{-1}(f(p))$.
\end{enumerate}
We call an $\epsilon$-almost Lipschitz submersion is {\it regular}, if in addition, 
\begin{enumerate}
	\item[(iii)] for any $y,z\in Y$, there are points $p\in f^{-1}(y),q\in f^{-1}(z)$ such that $|\theta(p,q)-\frac{\pi}{2}|\le \epsilon.$
\end{enumerate}

\begin{theorem}[Lipschitz submersion \& fibration] \label{thm:Lipschitz-submersion}
	For any dimension $n$ and positive number $\mu_0$, there exist positive numbers $\delta(n)$ and $\epsilon(n,\mu_0)$ such that for any $m$-dimensional Alexandrov space $X$ with curv $\ge -1$ and any $n$-dimensional Alexandrov space $Y$ with curv $\ge -1$, if
	\begin{enumerate}\numberwithin{enumi}{theorem}
		\item the $\delta$-strained radius of $Y$, $r_{\text{$\delta$-str}}(Y)\ge \mu_0$ with $0<\delta<\delta(n)$, and
		\item the Gromov-Hausdorff distance $d_{GH}(X,Y)\le
		\epsilon<\epsilon(n,\mu_0)$, 
	\end{enumerate}
	then there exists a regular $\varkappa(\delta,\epsilon|n)$-almost Lipschitz submersion $f : X \to Y$ that is a Hurewicz fibration.
\end{theorem}
\begin{remark}
	If in addition, every $f$-fiber is a topological manifold
	without boundary of co-dimension $n$, then $f$ is a locally trivial fibration; see \cite{RX12}.
	
	We also prove that the fibration in Theorem \ref{thm:Lipschitz-submersion} is uniquely determined in the homotopic sense; see Theorem \ref{thm:uniqueness}.
\end{remark}

Theorem \ref{thm:Lipschitz-submersion} can be traced back to the fibration theorem \cite{Fu87}, \cite{Ya91}, \cite{Pr93} for manifolds, which has played a fundamental role in the study of collapsed manifolds. The existence of regular almost Lipschitz in Theorem \ref{thm:Lipschitz-submersion} is due to Yamaguchi \cite{Ya96}, where he conjectured that it should be a locally trivial fibration. Here we partly verify his conjecture.

\begin{remark}
	A direct corollary of Theorem \ref{thm:Lipschitz-submersion} is a long exact sequence
	arising from the fibration:
	\begin{equation}\label{long-exact-seq}
	\begin{aligned}
	&\cdots\to\pi_l(F,x)\to
	\pi_l(X,x)\overset{f_*}{\to}\pi_l(Y,f(x))\to\pi_{l-1}(F,
	x)\to\cdots\to \\
	&\cdots\to\pi_1(Y,f(x))\to 0.
	\end{aligned}
	\end{equation}
	In \cite{Pr97} Perelman concluded the same long exact sequence 
	under a much weaker situation, that is, when a sequence of Alexandrov spaces $X_i$ with curv $\ge\kappa$ collapses to a limit space $Y$, if $Y$ contains no proper extremal subsets, then (\ref{long-exact-seq}) holds for $i$ large and a regular fiber $F$ (i.e., the fiber of a lifting map to $X_i$ of regular admissible maps locally defined on $Y$ to $\Bbb R^n$, see 	\cite{Pr97}). 
	
	By the proof of Theorem \ref{thmb-margulis}, both the homotopy fiber in Theorem \ref{thm:Lipschitz-submersion} and Perelman's regular fiber admit a $w(m-n)$-nilpotent fundamental group, where $w$ depends on the codimension.
\end{remark}

The gradient push developed by Kapovitch-Petrunin-Tuschmann \cite{KPT10} is important for us to deduce the uniform index bound in Theorem \ref{thmb-margulis}, as what happened for almost nonnegatively curved manifolds in \cite{KPT10}.
\begin{theorem}[Gradient push, {\cite[Lemma 2.5.1]{KPT10}}]\label{thm-gradientpush-1}
	There are $\delta(n), T(n)>0$ such that if the metric ball $B_1(p_0)$ centered at $p_0$ of radius $1$ is relative compact in an Alexandrov $n$-space $X$ with curvature $\ge -1$, then there are regular points $\{a_j,b_j\}_{j=1}^n$ and $q_0$ in $B_{\frac{1}{100}}(p_0)$ such that $\{a_j,b_j\}_{j=1}^n$ is a $(n,\delta)$-strainer at $q_0$ and any point $q$ in $B_{\delta |a_nb_n|}(q_0)$ can be pushed successively by the gradient flows of $\frac{1}{2}\dist_{q_0}^2$, $\frac{1}{2}\dist_{a_j}^2$, $\frac{1}{2}\dist_{b_j}^2$ $(j=1,\dots, n)$ to any point $p\in B_{\frac{1}{2}}(q)$ in total time $\le T(n)$.
\end{theorem}

Compared to the case of manifolds, a crucial difference on an Alexandrov $n$-space $X$ is that, there may be proper extremal subsets and no gradient curves can get out of them. 
When pushing a loop at a regular point to another regular point, it is a subtle point whether the successive gradient curve at base point do not pass any proper extremal subset in $X$.

Since it is hard for us by following \cite{KPT10} to check this directly, in the appendix we give a detailed proof of Theorem \ref{thm-gradientpush-1}, by constructing a specific gradient pushing broken line, which consists of $k$-regular (i.e., the tangent cone $T_pX$ at least splits off $\mathbb R^k$) or $(n,\delta)$-strained points when $a_j,b_j$ and the ending point $p$ are $k$-regular. In particular, the gradient push between regular points can always keep away from extremal subsets. We also sharpen the universal time bound $T(n)$ to $n^2 \delta^{-1}$, improving the universal time bound $\delta^{-n^2}$ in \cite{KPT10}. This provides a detailed justification for the gradient push in proving the Margulis lemma on an Alexandrov space.

\begin{remark}
	Kapovitch-Wilking \cite{KW11} developed a replacement (see the zooming in property and rescaling theorem in \cite{KW11}) of Yamaguchi's fibration theorem \cite{Ya91} and gradient push \cite{KPT10} in proving the Margulis lemma for manifolds with lower bounded Ricci curvature.
	
	Note that it is necessary to change base points many times when the rescaling theorem is applied. Since a fixed base point is chosen to be valid for our case at every scale, the proof of Theorem \ref{thmb-margulis} is more direct than \cite{KW11}. 
\end{remark}

Now let us briefly explain ideas of the proofs. According to \cite{KPT10}, a finite generated group $G$ is $w$-nilpotent, if it admits a filtration $G_1=G\vartriangleright G_2 \vartriangleright \cdots \vartriangleright G_l=\{e\}$, where $l\le n$, each $G_i\vartriangleleft G_1$, $G_i/G_{i+1}$ is $c$-abelian, and the conjugate action of $G_1$ on $G_i/G_{i+1}$, namely $\rho_i:G_1\to \operatorname{Out}(G_i/G_{i+1})$, has a finite image, whose order is bounded by $C$. By a contradicting argument and an iterated blowing-up process, we will prove that around any $p\in X$, there is a nearby point $q$ at which the local fundamental group corresponding to different collapsing scales (see Definition \ref{def-local-group}) has a filtration as above. Then Theorem \ref{thmb-margulis} follows from a compact packing argument as in \cite{KW11}. Theorem \ref{thm:Lipschitz-submersion} is used in proving $G_{i+1}\vartriangleleft G_i$ (for an alternative proof, see \cite{FY92} or \cite{Ya96}). The normal property $G_{i}\vartriangleleft G_1$ and a uniform bound on $\# \rho_i(G_1)$ follow from the universal time bound in Theorem \ref{thm-gradientpush-1}.

According to Ferry's result (\cite{Fe78}, see also Theorem
\ref{thm:strong-regular}), the homotopy lifting property
holds for the map in Theorem \ref{thm:Lipschitz-submersion} if there are controlled homotopy
equivalences between nearby fibers (called {\it strong regular}, see
Section 2.2) and all fibers are abstract neighborhood retracts.
As a generalization of the tubular neighborhood of fibers and horizontal curves of an $\epsilon$-Riemannian submersion, a neighborhood retraction $\varphi_p$ to a fiber 
$f^{-1}(p)$ of a LcL was constructed in \cite{RX12} (see also Proposition
\ref{prop:gradient-retract}, Section 2.4), which is defined via iterated gradient deformations of distance functions.
By this neighborhood retraction associated to every fiber, we are able to define controlled
homotopy equivalences between nearby fibers
and prove the fiber is locally contractible.

The remaining of the paper is divided into three parts. In Section 2, we will review some topological results and prove Theorem \ref{thm:Lipschitz-submersion}. In Sections 3,4 and 5 we prove Theorem \ref{thmb-margulis}. In the Appendix we give an elementary construction of the gradient push in Theorem \ref{thm-gradientpush-1} with a sharpened time estimate improving that in \cite{KPT10}.

\begin{acknowledgements}
	The first author would like to thank Xiaochun Rong and Hao Fang for helpful discussions, and thank the University of Iowa for hospitality and support during a visit in which a part of the work was completed. The second author would like to thank Yin Jiang and Liman Chen for helpful suggestions. We are grateful to Fuquan Fang for pointing out the nilpotency result in \cite{BGT2012} to us. This work is supported partially by National Natural Science Foundation of China [11871349], [11821101], by research funds of Beijing Municipal Education Commission and Youth Innovative Research Team of Capital Normal University.
\end{acknowledgements}

\section{Homotopy lifting properties}

\subsection{Proof of Theorem \ref{thm:Lipschitz-submersion}}

A map $f:X\to Y$ between two metric spaces is called an {\it
	$e^\epsilon$-Lipschitz and co-Lipschitz} \cite{Ka07}, \cite{RX12} (briefly,
$e^\epsilon$-LcL), if for any $p\in X$, and any $r>0$, the
metric balls satisfy 
\begin{equation}
B_{e^{-\epsilon}r}(f(p))\subseteq
f(B_r(p))\subseteq B_{e^\epsilon r}(f(p)).\tag{1.7}
\end{equation}
A $1$-LcL preserves metric balls exactly and is called a {\it
	submetry} \cite{BG00}. Clearly, a regular $\epsilon$-almost Lipschitz submersion is an $e^{C\epsilon}$-LcL for some universal constant $C$.

Since by definition, a regular almost Lipschitz submersion satisfies the LcL property, it suffices to show Theorem \ref{thm-Hurewicz} below.

In order to simplify constant dependence, we introduce another terminology other than $\delta$-strained radius. 

An $n$-dimensional Alexandrov 
space $Y$ is called {\it $\epsilon$-almost Euclidean} if for any point $p\in Y$, there is a 
neighborhood $U$ containing $p$ and a bi-Lipschitz map $\varphi:U\to 
\varphi(U)\subset \Bbb R^n$ onto an open neighborhood in $\Bbb R^n$ 
such that for any $x,y\in U$,
\begin{equation}\label{eq:almost-Euclidean}
e^{-\epsilon}|xy|\le |\varphi(x)\varphi(y)|\le e^\epsilon 
|xy|.\tag{2.1}
\end{equation}
If (\ref{eq:almost-Euclidean}) holds on every $r$-ball in $Y$, then $Y$ is called {\it $(r,\epsilon)$-almost Euclidean}.
By \cite[Theorem 5.4]{BGP92}, an Alexandrov space with curv $\ge -1$ and $\delta$-strained radius $\ge \mu_0$ is $(\mu_0,\varkappa(\delta|n))$-almost Euclidean.
\begin{theorem}\label{thm-Hurewicz}
	Let $f:X\to Y$ is a $\sqrt{1.023}$-LcL between finite-dimensional 
	Alexandrov spaces with curv $\ge \kappa$. If $f$ is proper and  
	the base space $Y$ is  
	$\ln\sqrt{1.023}$-almost Euclidean, then $f$
	is a Hurewicz fibration, i.e., satisfying the
	homotopy lifting property with respect to any space.
\end{theorem}
Theorem \ref{thm-Hurewicz} has appeared in an earlier preprint \cite{Xu12}.

\vspace{2mm}
\begin{proof}[Proof of Theorem \ref{thm:Lipschitz-submersion}]
	~
	
	The existence of a regular almost Lipschitz submersion is proven by Yamaguchi \cite{Ya96}. By Theorem \ref{thm-Hurewicz} and the discussion above, any regular almost Lipschitz submersion $f:X\to Y$ is a Hurewicz fibration.
\end{proof}

The remaining of this section is devoted to prove Theorem \ref{thm-Hurewicz}.

\subsection{A sufficient condition for a fibration}
The following topological results are used in the proof of Theorem \ref{thm-Hurewicz}. 

For any Hurewicz fibration $f:X\to Y$, if
$Y$ is path-connected, then by definition the fibers are homotopy
equivalent to each other. In \cite{Fe78} Ferry proved that the
inverse is also true, if the homotopy equivalences between nearby
fibers and the homotopies are under control in the following sense.

A map $f:X\to Y$ between metric spaces is said to
be {\it strongly regular} \cite{Fe78} if $f$ is proper and if for each
$p\in Y$ and any $\epsilon>0$ there is a $\delta>0$ such that if
$|pp_1|<\delta$,
then there are homotopy equivalences between fibers $\varphi_{pp_1}:
f^{-1}(p)\to f^{-1}(p_1)$, $\varphi_{p_1p}: f^{-1}(p_1)\to f^{-1}(p)$
which togther with the homotopies move points in distance $<\epsilon$.

A topological space $X$ is an {\it absolute neighborhood retract}
(ANR) if there is an embedding of $X$ as a closed subspace of the
Hilbert cube $I^{\infty}$ such that some neighborhood
$N$ of $X$ retracts onto $X$. If $X$ is finite covering dimensional
and locally contractible, then $X$ is an ANR (\cite{Bo55}).

\begin{theorem}[\cite{Fe78}]\label{thm:strong-regular}
	If $f:E\to B$ is a strongly regular map onto a complete finite
	covering dimensional space $B$ and all fibers are ANRs, then $f$ is a
	Hurewicz fibration.
\end{theorem}

\begin{remark}\label{rem-local-fibration}
Note that the properties of being an ANR or a Hurewicz fibration are local properties (cf. \cite{Fe78}), Theorem \ref{thm:strong-regular} was proved locally in \cite{Fe78}.
Moreover, the Lipschitz submersion in Theorem \ref{thm:Lipschitz-submersion} can be constructed locally (\cite{Ya96}).  Hence, both of them holds over $\epsilon$-almost Euclidean points in a complete Alexandrov space. And so are Theorem \ref{thm-Hurewicz} and Theorem \ref{thm:Lipschitz-submersion}. 
\end{remark}

According to Theorem \ref{thm:strong-regular} and the discussion above, Theorem \ref{thm-Hurewicz} holds if an $e^\epsilon$-LcL between Alexandrov spaces with almost Euclidean base space is strongly regular, and all its fibers are locally contractible. 

\subsection{Gradient estimate for an LcL}
Let us first recall a basis property of an
$e^\epsilon$-LcL $f:X\to Y$.
For any compact subset $S\subset Y$, let $\operatorname{dist}_S$ be
the distance function to $S$ in $Y$,
$$\operatorname{dist}_S(y)=\abs{yS}=\inf\{d(y,s): s\in S\}.$$
Then the two functions $\operatorname{dist}_S\circ f$ and
$\operatorname{dist}_{f^{-1}(S)}: X\to
\Bbb R_+$ satisfy (see Lemma 1.4 in \cite{RX12})
\begin{equation}\label{eq:lcl}
e^{-\epsilon}\cdot \operatorname{dist}_S\circ f \le
\operatorname{dist}_{f^{-1}(S)}\le e^\epsilon \cdot
\operatorname{dist}_S\circ f.\tag{2.3}
\end{equation}

Since LcL property is rescaling invariant, from now on we assume that $X$ is an Alexandrov space with curv $\ge -1$, $Y$ is an $n$-dimensional Alexandrov space with curv $\ge -1$ that is $\epsilon$-almost Euclidean. Let
$f:X\to Y$ be an $e^\epsilon$-LcL. 
Under the assumption that $Y$ is a Riemannian manifold,
we constructed in \cite{RX12} a neighborhood retraction $\varphi_p$ 
of $f$-fiber over $p\in Y$, which
is continuously depending on $p$ and can be used as a weaker
replacement of the horizontal lifting of minimal geodesics.
In the proof of Theorem \ref{thm-Hurewicz} we will apply it
to define controlled homotopy equivalences between nearby fibers.
Because now $Y$ is an Alexandrov space, 
for reader's convenience we recall
its construction and point out the differences to \cite{RX12} in below.

For an $\epsilon$-almost Euclidean point $p\in Y$, let $r_p$ denote the maximal number that there is a map 
$\varphi: B(p,r_p)\to \Bbb R^n$ satisfying (\ref{eq:almost-Euclidean}). Let $S_r(p)=\partial B_r(p)$ be the metric sphere 
around $p$ and let $x$ be any point in $B_r(p)\setminus\{p\}$.
We have the following estimate on the gradient of distance function
$\operatorname{dist}_{f^{-1}(S_r(p))}$.

\begin{lemma}\label{lem-gradient-estimate} Let $f:X\to Y$ and $0<r\le \min\{r_p,1\}$ be as above.
		Let $x$ be point in $f^{-1}(B_r(p))\setminus f^{-1}(p)$. 
		The gradient vector of $\operatorname{dist}_{f^{-1}(S_r(p))}$ 
		satisfies
		\begin{equation}\label{eq:gradient-estimate-totalspace}
		1\ge \abs{\nabla_x
			\operatorname{dist}_{f^{-1}(S_r(p))}}
		\ge  1 -
		(e^{2\epsilon}-1) \cdot \frac{2r^2}{\abs{x f^{-1}(p)}\cdot
			\abs{x f^{-1} (S_r(p))}}.\tag{2.4}
		\end{equation}
\end{lemma}

\begin{proof}	
	
	The proof is similar to Lemma 1.5 (1.5.1) in \cite{RX12}.
	Let $z\in f^{-1}(S_r(p))$, $y\in f^{-1}(p)$ be such that
	$|xz|=|x f^{-1}(S_r(p))|$ and $|xy|=|xf^{-1}(p)|$. 
	Let $v$ be the direction at $x$ of a minimal geodesic from $x$ to $y$.
	It suffices to bound 
	$\cos \measuredangle (v,w)$ from above for any direction $w$ from $x$ 
	to $f^{-1}(S_r(p))$.
	
	Since $f$ and $\varphi$ are $e^{\epsilon}$-LcLs, by (\ref{eq:lcl}) we 
	directly see
	$$\aligned  |xy| &\le e^{2\epsilon}\cdot |\varphi(f(x))\varphi(p)|, 
	\\ 
	|xz| &\le e^{2\epsilon}\cdot |\varphi(f(x))\varphi(S_r(p))|,\\
	|yz| &\ge |z f^{-1}(p)| \ge e^{-\epsilon}|f(z)p|=e^{-\epsilon}\cdot 
	r. \endaligned $$
	Moreover, 
	\begin{align*}
	|\varphi(f(x))\varphi(p)|+|\varphi(f(x))\varphi(S_r(p))|&\le 
	|\varphi(f(x))\varphi(p)|+ |\varphi(f(x))S_{e^\epsilon 
		r}(\varphi(p))|\\
	&= e^\epsilon r.
	\end{align*}
	Thus
	$$|yz|\le |xy|+|xz|\le 
	e^{2\epsilon}\cdot 
	(|\varphi(f(x))\varphi(S_r(p))|+|\varphi(f(x))\varphi(S_r(p))|)\le 
	e^{3\epsilon}r.
	$$
	
	Since the proof below is similar for different curvature lower bound, for simplicity we only prove for $\kappa=0$. By the Euclidean cosine law, we derive
	$$\aligned \cos \tilde \measuredangle_0 (zxy)&=\frac {|xz|^2+|xy|^2-
		|yz|^2}{2|xz|\cdot |xy|}\\
	&=\frac{(|xz|+|xy|)^2-|yz|^2}{2|xz|\cdot |xy|}-1\\
	&= \frac{(|xz|+|xy|-|yz|)\cdot (|xz|+|xy|+|yz|)}{2|xz|\cdot|xy|}-1\\
	&\le (e^{2\epsilon}-1) \cdot
	\frac{r^2}{|xz|\cdot|xy|} -1.
	\endaligned$$
\end{proof}

By Lemma \ref{lem-gradient-estimate} and a standard argument, for sufficient small $\epsilon$ ($e^{2\epsilon}\le
1.02368$), points in $f^{-1}(B_{\frac{2r}{3}}(p))$ can be flowed into
$f^{-1}(B_{\frac{r}{3}}(p))$ along gradient curves of
$\operatorname{dist}_{f^{-1}(S_r(p))}$ in a definite time.

\begin{lemma}[Lemma 1.5 in \cite{RX12}]\label{lem:gradient-estimate}
	For any $p\in Y$ and $r<\min\{r_p,
	\frac{1}{2e^\epsilon}\}$, there is a
	constant  $C_0(\epsilon)>0$ depending on $\epsilon$ such that for all
	$x\in f^{-1}(B_{\frac {2r}3}(p))$,  the gradient
	curve $\Phi(t,x)$ of the function
	$\operatorname{dist}_{f^{-1}(S_r(p))}$
	satisfies
	$$\Phi(x,t)\in f^{-1}(B_{\frac {r}3}(p)),
	\qquad t\ge C_0^{-1}\cdot\left(\frac{2}{3}e^{\epsilon} r
	-\abs{x f^{-1}(S_r(p))}\right).$$
\end{lemma}

\subsection{Neighborhood retraction of a fiber}
In this part we construct a neighborhood retraction around a fiber $f^{-1}(p)$ which continuously depends on $p$.

We first define a gradient deformation of
$\operatorname{id}_{f^{-1}(B_{\frac{2r}{3}}(p))}$ which maps
$f^{-1}(B_{\frac{2r}{3}}(p))$ into $f^{-1}(B_{\frac{r}{3}}(p))$ and
fixes $f^{-1}(B_{0.3r}(p))$.
Let
$$T_{p,r}(x)=\max\left\{0,C_0^{-1}\cdot\left(\frac{2}{3}e^{\epsilon}
r -\abs{x f^{-1}(S_r(p))}\right)\right\},$$
and $\Phi_p^{T_{p,r}(x)}(x)=\Phi(x,T_{p,r}(x))$ be the gradient
deformation of $\operatorname{id}_{f^{-1}(B_{\frac{2r}{3}}(p))}$ 
with respect to $\operatorname{dist}_{f^{-1}(S_r(p))}$.
Then by Lemma \ref{lem:gradient-estimate} and direct calculation, for $e^{2\epsilon}\le 1.02368$
and $r<\min\{r_p,
\frac{1}{2e^\epsilon}\}$, we have

\begin{equation}\label{eq:gradient-flow}
\begin{cases}
\Phi_p^{T_{p,r}(x)}(x)\in f^{-1}(B_{\frac {r}3}(p)), &\forall\; x\in
f^{-1}(B_{\frac{2r}{3}}(p)),\\
T_{p,r}(x)=0, &\forall\; x\in
f^{-1}(B_{0.3 r}(p)).\tag{2.5}
\end{cases}
\end{equation}

In \cite[Proposition 1.6]{RX12} we proved that $\Phi_p^{T_{p,r}(x)}(x)$ is continuous both in $p$ and $x$, provided that $Y$ is a Riemannian manifold and $r$ is smaller than the injectivity radius of $Y$. In the following we prove the same holds when $Y$ is an almost Euclidean Alexandrov space.

\begin{lemma}\label{lem-continuity} Let $0<\epsilon<\ln \sqrt{1.02368}$, and let $f:X\to Y$ be an $e^\epsilon$-LcL between Alexandrov spaces such that $Y$ is $(\mu_0,\epsilon)$-almost Euclidean. Then for any $0<r<\frac{1}{2}\cdot \min\{\mu_0,1\}$,
$$ \Psi:\bigcup_{p\in Y} \{p\}\times
f^{-1}(B_{\frac{2r}{3}}(p))\subset Y\times X
\to X,\quad \Psi(p,x)=\Phi_p^{T_r(x)}(x)$$
is a continuous map.
\end{lemma}
\begin{proof}
	Since the proof is similar to \cite[Proposition 1.6]{RX12}, 
	we give a sketch proof by pointing out the difference.

	Because the gradient curves are stable as function converges
	(\cite{Petr07}), it suffices to show that the distance functions $\dist_{f^{-1}(S_r(p))}, \dist_{f^{-1}(S_r(q))}$ (to $f^{-1}(S_r(p))$ and $f^{-1}(S_r(q))$ respectively) are $C|pq|$-close for small $|pq|$ and a constant $C$. 
	
	By the definition of LcL, it is easy to verify (see \cite[Lemma 1.4, Lemma 1.7]{RX12}) that the Hausdorff distance and the difference between $\dist_{f^{-1}(S_r(p))}$ and $\dist_{f^{-1}(S_r(q))}$ satisfy 
	\begin{align}
	d(\dist_{f^{-1}(S_r(p))}, \dist_{f^{-1}(S_r(q))})&=d_H(f^{-1}(S_r(p)), f^{-1}(S_r(q))),\tag{\ref{lem-continuity}.1}\\
	d_H(f^{-1}(S_r(p)), f^{-1}(S_r(q)))&\le e^\epsilon\cdot d_H(S_r (p), S_r(q)).\tag{\ref{lem-continuity}.2}
	\end{align}
	
	Let $d(p,q)=\varepsilon_1$. By (\ref{lem-continuity}.1) and (\ref{lem-continuity}.2), what remains is to show $S_r(p)$ and $S_r(q)$ are $C\epsilon_1$-close in Hausdorff distance.
	
	Let $z$ be a middle point in a minimal geodesic $[pq]$. Since both $S_r(p)$ and $S_r(q)$ lie in the annulus $B_{r+\varepsilon_1}(z)\setminus B_{r-\varepsilon_1}(z)$, it is easy to see that one only needs to bound the Hausdorff distance between metric spheres $S_{r+\varepsilon_1}(z)$ and $S_{r-\varepsilon_1}(z)$, i.e., for some constant $C$, $d_H(S_{r-\varepsilon_1}(z),S_{r+\varepsilon_1}(z))\le C\epsilon_1$. 
	
	Indeed, for any point $x\in S_{r+\varepsilon_1}(z)$, since the point $x_1$ in a minimal geodesic $[xz]$ with distance $|x_1x|=2\varepsilon_1$ lies in $S_{r-\varepsilon_1}(z)$, $S_{r+\varepsilon_1}(z)$ lies in $2\varepsilon_1$-neighborhood of $S_{r-\varepsilon_1}(z)$. 
	
	Conversely, let $x\in S_{r-\varepsilon_1}$. By the proof of Lemma \ref{lem-gradient-estimate}, there is a point $y$ in $S_{2r}(z)$ such that the comparison triangle $\tilde\measuredangle_{-1}(zxy)$ is larger than $\pi/2$ by a positive definite error $\theta>0$.
	By the triangle version of Toponogov theorem, there exists $y_1$ in $[xy]$ with distance $|xy_1|\le \frac{2\varepsilon_1}{-\cos \tilde\measuredangle_{-1}(zxy)}$ such that $|y_1z|=r+\varepsilon_1$.
\end{proof}

Next, let us repeat the construction above for the sequence
$\{r_i=\frac{r}{2^i}\}_{i=0,1,2,\cdots}$
and let $\Phi_{p,i}^{T_{p,i}}(x)=\Phi_{p,i}(x,T_{p,r_i}(x))$ be the
gradient curves of $\operatorname{dist}_{f^{-1}(S_{r_i}(p))}$ at $x$
with time $T_{p,r_i}(x)$. By (\ref{eq:gradient-flow}),
$\Phi_{p,i}^{T_{p,i}}:f^{-1}(B_{r_i}(p))\to X$
takes $f^{-1}(B_{\frac{2}{3}\cdot \frac{r}{2^i}}(p))$ into 
$f^{-1}(B_{\frac{1}{3}\cdot\frac{r}{2^{i+1}}}(p)),$
and $$\left .\Phi_{p,i}^{T_{p,i}}
\right |_{f^{-1}(B_{0.3\frac{r}{2^i}}(p))} = \operatorname{id}.$$
Hence the iterated gradient deformations
$$\Phi_{p,i}^{T_{p,i}}\circ \Phi_{p,i-1}^{T_{p,i-1}}\circ \cdots
\Phi_{p,0}^{T_{p,0}}$$
is well-defined on $f^{-1}(B_{\frac{2r}{3}}(p))$ and its restriction
on $f^{-1}(B_{0.3\frac{r}{2^i}}(p))$ is identity.
Because $$T_{p,r_i}(x)\le \frac{r}{2^{i-1}}\cdot
\frac{e^\epsilon}{3}\cdot C_0^{-1},$$ it can be
directly verified that the sequence of maps
\begin{gather*}
\Psi_i:\bigcup_{p\in Y} \{p\}\times f^{-1}(B_{\frac{2r}{3}}(p))
\to X,\\
\Psi_i(p,x)=\Phi_{p,i}^{T_{p,i}}\circ \Phi_{p,i-1}^{T_{p,i-1}}\circ
\cdots \Phi_{p,0}^{T_{p,0}}(x)
\end{gather*}
uniformly converges. The limit $\varphi_p(x)=\displaystyle\lim_{i\to
	\infty}\Psi_i(p,x)$ gives a retraction
from the neighborhood $f^{-1}(B_{\frac{2r}{3}}(p))$ to $f^{-1}(p)$,
which by Lemma \ref{lem-continuity} is continuous both in $p$ and $x$. We summarize it to the following proposition.

\begin{proposition}	\label{prop:gradient-retract}
	For any $0<r<\frac{1}{2}\min\{\mu_0,1\}$, there is a deformation retraction
	$\varphi_p(x)$ from a neighborhood $f^{-1}(B_{\frac{2r}{3}}(p))$ to
	the fiber $f^{-1}(p)$ such that  
	$$\varphi: \bigcup_{p\in Y} \{p\}\times
	f^{-1}(B_{\frac{2r}{3}}(p))\to X, \quad \varphi(p,x)=\varphi_p(x)$$
	is continuous both in $p$ and $x$, and 
	satisfies 
	\begin{enumerate} \numberwithin{enumi}{theorem}
		\item $\varphi_p(x)=x$ for any $x\in
		f^{-1}(p)$, and 
		\item $\abs{x \varphi_p(x)} \le
		2C_1r,
		$ for some constant
		$C_1(\epsilon)$ depending only on $\epsilon$.
	\end{enumerate}
\end{proposition}

\vspace{2mm}
\begin{proof}[Proof of Theorem \ref{thm-Hurewicz}]
	~
	
	Up to a rescaling we assume that the lower curvature bounds of both $X$ and $Y$ are 	$-1$.
	By Theorem \ref{thm:strong-regular}, it suffices to
	show that $f$ is strong regular and any fiber is an ANRs. 
	For any $p,q\in B$ with small distance
	$0<\abs{pq}<\frac{1}{2}\min\{r_p,
	\frac{1}{2e^\epsilon}\}$, let $\rho=2\abs{pq}$.
	By the definition of LcL, it is easy to see that
	$$e^{-\epsilon}\cdot |pq| \le
	d_H(f^{-1}(p),f^{-1}(q))\le e^\epsilon\cdot |pq|.$$
	Thus $f^{-1}(q)$ lies in $e^\epsilon \frac{\rho}{2}$-neighborhood of $f^{-1}(p)$ and vice versa.  
	By Proposition
	\ref{prop:gradient-retract}, there are neighborhood
	retractions $\varphi_p:
	f^{-1}(B_{\frac{2\rho}{3}}(p))\to f^{-1}(p)$
	and $\varphi_q: f^{-1}(B_{\frac{2\rho}{3}}(q))\to f^{-1}(q)$
	around $f^{-1}(p)$ and $f^{-1}(q)$ respectively.
	Then the homotopy equivalences between fibers can be chosen to be
	$\left.\varphi_p\right|_{f^{-1}(q)} :f^{-1}(q)\to f^{-1}(p)$
	and $\left.\varphi_q\right|_{f^{-1}(p)}:f^{-1}(p)\to f^{-1}(q)$, and
	the homotopies are
	$H_t=\varphi_p\circ \varphi_{\gamma(t)}:f^{-1}(p)\to f^{-1}(p)$ and
	$K_t=\varphi_q\circ \varphi_{\gamma(1-t)}:f^{-1}(q)\to f^{-1}(q)$,
	where $\gamma:[0,1]\to B$ is a minimal geodesic from $p$ to $q$.
	By (\ref{prop:gradient-retract}.2), $\abs{H_t(x)x}\le 4C_1\rho$ and
	$\abs{K_t(x)x}\le 4C_1\rho$. Therefore $f$ is strongly regular.
	
	According to \cite{Pr91} (cf. \cite{Ka07}, \cite{Petr07}), an Alexandrov space with curvature bounded below is
	locally contractible. For $x\in
	f^{-1}(p)$, let $U_x\ni x$ be a contractible neighborhood around
	$x$ and $H_t:U_x\to U_x$ be the homotopy from
	$\operatorname{id}_{U_x}$ to the retraction $r:U_x\to \{x\}$ such
	that $H_t(x)=x$. Then $\varphi_p\circ H_t$ is a homotopy
	from $\operatorname{id}_{U_x\cap f^{-1}(p)}$ to the retraction $r
	:U_x\cap f^{-1}(p)\to \{x\}$. Therefore $f^{-1}(p)$ is locally
	contractible and thus an absolute neighborhood retract.
\end{proof}

\subsection{Homotopic uniqueness of fibration}

Recently it is proved in \cite{JX18} that two collapsed metrics $g_{i}$ ($i=0,1$) on $M$ induces the same nilpotent Killing structure up to a diffeomorphism, provided $g_{i}$ are $L_0$-Lipschitz equivalent and sufficiently collapsed. 

In the following we prove that in the homotopic sense, the collapsing fibration in Theorem \ref{thm:Lipschitz-submersion} is unique.

We say that two Hurewicz fibrations $f_i:X_i\to Y$ ($i=0,1$) are
{\it fibrewise homotopy equivalent} if there are fiber-preserving maps 
$h:X_0 \to X_1$ and $g : X_1 \to X_0$ and fiber-preserving homotopies 
between $g\circ h$ and identity $1_{X_0}$, and between $h\circ g$ and
$1_{X_1}$.
We say that Hurewicz fibrations $f_i:X_i\to Y_i$ ($i=0,1$) are {\it 
	equivalent} if there is a homeomorphism $\psi:Y_0\to 
Y_1$ such that $\psi\circ f_0:X_0\to Y_1$ is fiber-homotopy equivalent 
to $f_1:X_1\to Y_1$. 

\begin{theorem}\label{thm:uniqueness}
	Let $X$, $Y_i$ $(i=0,1)$ be Alexandrov spaces with curv $\ge -1$ such that $Y_i$ satisfies (1.1.1), the dimension $\dim Y_1=\dim Y_2$, and (1.1.2) holds for $d_{GH}(X,Y_i)$. Then any two Hurewicz fibrations $f_i$ from $X$ to $Y_i$ $(i=0,1)$ provided by Theorem \ref{thm:Lipschitz-submersion} are equivalent.
\end{theorem}

	It follows either from \cite[Theorem 9.8]{BGP92} (a key lemma of its proof has a flaw, for a correct proof see \cite{SSW13}), or from Theorem \ref{thm:Lipschitz-submersion}, that there is $e^{\varkappa(\epsilon|n)}$-bi-Lischitz map $\varphi:Y_0\to Y_1$ such that $\varphi\circ f_0$ is $100\epsilon$-close to $f_1$. Thus, the uniqueness in Theorem \ref{thm:uniqueness} is reduced to a stability result below.

\begin{proposition}\label{prop-stability}
	Let $X$ and $Y$ be two Alexandrov spaces with curv 
	$\ge\kappa$, where $Y$ is $(\mu_0,\ln\sqrt{1.023})$-almost Euclidean. If two $\sqrt{1.023}$-LcLs $f_{0},f_1:X\to Y$ are $\frac{\mu_0}{3}$-close, i.e., 
	\begin{equation}
	d(f_0,f_1)=\sup_{x\in X}|f_0(x)f_1(x)|<\frac{1}{3}\mu_0,  \tag{2.2}
	\end{equation}
	then they are equivalent as Hurewicz fibrations.
\end{proposition}
Theorem \ref{prop-stability} is an improvement of a stability result in \cite{Xu12}.

In the proof of Proposition \ref{prop-stability}, we need a ``canonical'' pointed contraction on the base space $Y$, which are constructed similarly as in Proposition \ref{prop:gradient-retract}. 
\begin{lemma}\label{lem-ptcontract-base}
	Let $Y$ be a $(\mu_0,\ln 1.02368)$-almost Euclidean Alexandrov space with curv $\ge -1$. 
	There is a continuous pointwise contraction on $Y$, 
	$$\tau: \bigcup_{p\in Y} \{p\}\times
	B_{\mu_0/2}(p)\times [0,1]\to Y, \quad \tau(p,x,0)=x,\quad \tau(p,x,1)=p.$$
\end{lemma}
\begin{proof}
	Note that the estimates in Lemma \ref{lem-gradient-estimate} and \ref{lem:gradient-estimate} also holds for the distance function $\dist_{S_r(p)}$ for $0<r\le \min\{\mu_0,1\}$.
	Let $\psi(p,x,t)$ be the limit of iterated gradient flows of $\dist_{S_{r_i}(p)}$ for $r_i=2^{-i}\mu_0$ with time $t\in [0,T_{p,r_i}(x)]$, where $C_0$ is the constant in Lemma \ref{lem:gradient-estimate} and $T_{p,r_i}(x)=\max\{0, C_0^{-1}(\frac{2}{3}e^\epsilon r_i-|xS_{r_i}(p)|)\}$. Let $T(p,x)=\sum_{i=0}^\infty T_{p,r_i}(x)$. It follows from the proof of Proposition \ref{prop:gradient-retract} directly that the map $\tau(p,x,t)=\psi(p,x,tT(p,x))$ satisfies the requirement of Lemma \ref{lem-ptcontract-base}.
\end{proof}

\vspace{2mm}
\begin{proof}[Proof of Proposition \ref{prop-stability}]
	~
	
	Let $f_0,f_1:X\to Y$ be the $\sqrt{1.023}$-LcLs between Alexandrov spaces with a $(\mu_0,\ln\sqrt{1.023})$-almost Euclidean base $Y$. We now construct 
	fiber-preserving maps $h,g:X\to X$ and fiber-preserving homotopies 
	$g\circ h$ to the identity $1_X$ and from $h\circ g$ to $1_X$ as
	follows.
	
	For any point $x\in X$, let $p=f_0(x)\in B$, let $F_0(p)$ be
	the fiber $f^{-1}_0(p)$ and $F_1(p)=f^{-1}_1(p)$. Suppose that 
	$d(f_0,f_1)<\frac{\mu_0}{3}$. Then by (\ref{eq:lcl}), $F_0(p)$ lies in the
	$\sqrt{1.023}\frac{\mu_0}{3}$-neighborhood of $F_1(p)$. Let $\varphi_p$
	be the neighborhood retraction of $F_1(p)$ in Proposition \ref{prop:gradient-retract} with respect to $f_1$,  we define
	$h:X\to X$ by $h(x)=\varphi_p(x)=\varphi_{f_0(x)}(x)$. Then the
	continuous map $h:X\to X$ is globally defined and maps all fibers of
	$f_0$ into that of $f_1$. Similarly we define $g:X\to X$ through the
	neighborhood retraction of $f_0$-fibers such that $f_0\circ
	g=f_1$, where $g(x)=\psi_{f_1(x)}(x)$ and $\psi_{q}$ is the
	neighborhood retraction of $f_0^{-1}(q)$ with respect to $f_0$.
	
	Note that $f_1(\varphi_{f_0(x)}(x))=f_0(x)$, thus $$g\circ h(x)=\psi_{f_1(\varphi_{f_0(x)}(x))}(\varphi_{f_0(x)}(x))=\psi_{f_0(x)}\circ \varphi_{f_0(x)}(x).$$
	Moreover, since $\varphi_{f_1(x)}$ is a neighborhood retract to $F_1(f_1(x))$, $\varphi_{f_1(x)}(x)=x$. Similarly, $\psi_{f_0(x)}(x)=x$, and thus
	$$\psi_{f_0(x)}\circ \varphi_{f_1(x)}=\operatorname{1}_X:X\to X.$$
	For $p_0=f_0(x)$, let $p_1=f_1(x)$ and let $p_t=\tau(p_1,p_0,t)$ be the map provided by Lemma \ref{lem-ptcontract-base}. Then $p_t$ is a curve from $p_0$ to $p_1$ continuously depending on $x$ and $t$. We define the
	fiber-preserving homotopy $H_t:X\to X$ by
	$H_t(x)=\psi_{f_0(x)}\circ \varphi_{p_t}(x)$. Then
	$H:[0,1]\times X\to X$ is a $f_0$-fiber-preserving continuous map such that $H_0=g\circ h$
	and $H_1=1_{X}$. A fiber-preserving homotopy from $h\circ g$
	to the identity $1_X$ can be defined similarly.
\end{proof}

\section{Margulis lemma on Alexandrov spaces}
\subsection{Proof of Theorem \ref{thmb-margulis}}

Let $X$ be a locally complete Alexandrov $n$-space of curv $\ge -1$. Let $B_1(p)$ be the $1$-ball centered at some point $p\in X$.

It is well-known that sufficient away from where $X$ is non-complete, the global Toponogov comparison on Alexandrov space holds (\cite{BGP92}, \cite{Pl96}). 
To be precisely, there is a constant $C_T$, such that Toponogov comparison holds for any triangle in $B_{\frac{1}{C_T}}(p)$, provided that $B_1(p)$ is relative compact. According to the proof of global Toponogov theorem in \cite{LS12} or \cite{Wa18}, it is enough to choose $C_T=100$.

Note that in a locally complete local Alexandrov $n$-space of curv $\ge -1$, the convex hull of a triangle may not be bounded.
However, by the proof of Toponogov comparison (cf. \cite{LS12}, \cite{Wa18}), any contradicting triangle can be reduced successively to other ones, whose perimeters decay in a definite ratio to form a converging geometric progression, such that a contradiction can be derived in a neighborhood of the initial triangle whose radius is not more than $12$-times of the initial perimeter.

Due to the above discussion, the fundamental facts on a complete Alexandrov space will be freely applied locally in this section without further mention.

We first reduce Theorem \ref{thmb-margulis} to the following special case. For any $p\in X$ and $q\in B_1(p)$, $0<r\le 1-|pq|$, let $\Gamma_{p}(q;r)$ be the subgroup of the fundamental group $\pi_1(B_1(p),q)$ generated by loops at $q$ lying in $B_r(q)$. As before, we will use $d(p,q)$ or $|pq|$ to denote the distance between two points.

\begin{theorem}\label{thm-good-point}
	Suppose that $B_1(p)$ is relative compact in $X$. Then there are positive constants $\epsilon(n),w(n)>0$, both depending only on the dimension $n$, such that there is a ``good'' point $q\in B_{\frac{1}{2C_T}}(p)$ satisfying $\Gamma_p(q;\epsilon(n))$ is $w(n)$-nilpotent. 
\end{theorem}

For general points in $X$, we need the following result in \cite{KW11}.

\begin{lemma}[{\cite[Step 2 in \S7] {KW11}}]\label{lem-finite-index-of-uniform-small}
	\label{lem-packing-index}
	For any positive integer $n$ and $0<\epsilon\le \epsilon(n)$ with $\epsilon(n)$ in Proposition \ref{thm-good-point}, there is $L(\epsilon,n)>0$ such that the following holds.
	
	Let $X$ be an Alexandrov $n$-space of curv $\ge -1$, and  $p\in X$ be a point such that $B_1(p)$ is relative compact in $X$. Let $\Gamma=\left<\beta_1,\cdots,\beta_k:d(\beta_ip,p)\le \frac{1}{100C_TL(\epsilon,n)}\right>$ be a discrete subgroup of isometries of $X$ that acts freely. Then the subgroup 
	$$H=\left<g\in\Gamma: d(gx,x)\le \epsilon, \forall x\in B_{\frac{1}{2C_T}}(p)\right>$$ has finite index 
	$[\Gamma:H]\le (2k+1)^{L(\epsilon,n)}$.
\end{lemma}
Note that for any isometry $\gamma$ of $X$ which moves $p$ not farther than $\frac{1}{100C_T}$ but a point in $B_{\frac{2}{3C_T}}(p)$ farther than $\epsilon$, $\gamma$ should move a point in any maximal $\frac{\epsilon}{4}$-net of $B_{\frac{2}{3C_T}}(p)$ farther than $\frac{\epsilon}{4}$. Thus the total possibility of such isometries can be reduced to permutations of lattice, whose total number is under control by the relative volume comparison (see \cite{LiRong12}). By considering the naturally extended action of $\Gamma$ on the $m$ times direct product space by $X$ itself, the total number of cosets of $H$ can be counting via a wordlength-cutting-off argument with $$L(\epsilon,n)=\frac{(\operatorname{vol} B_{-1}^n(\frac{3}{4C_T}))^{m(\epsilon,n)}}{\operatorname{vol} B_{-1}^{m(\epsilon,n)n}(\frac{\epsilon}{8})}, \qquad \text{where}\quad m(\epsilon,n)=\frac{\operatorname{vol} B_{-1}^n(\frac{2}{3C_T})}{\operatorname{vol} B_{-1}^n(\frac{\epsilon}{8})},$$ 
and $B_{-1}(r)$ denotes a ball in the Hyperbolic space $\mathbb H^n$. For details, see \cite[\S 7]{KW11}.

Assuming Theorem \ref{thm-good-point}, we now prove Theorem \ref{thmb-margulis}.

\vspace{2mm}
\begin{proof}[Proof of Theorem \ref{thmb-margulis}]
	~
	
	Let $(\tilde X,\tilde p)\to (B_1(p),p)$ be the universal cover of $B_1(p)$.
	Now we take $\epsilon(n)$ and $q$ to be the constant and a corresponding ``good'' point given by Theorem \ref{thm-good-point}. Let $$H=\left<g\in \pi_1(B_1(p),p):d(gx,x)\le \epsilon(n) \text{ for any $x\in B_{\frac{1}{2C_T}}(\tilde p)$}\right>.$$
	
	Since $q\in B_{\frac{1}{2C_T}}(p)$, $H$ can be viewed as a subgroup of $\Gamma_p(q;\epsilon(n))$, hence $H$ is $w(n)$-nilpotent.
	
	Since for some $0<\delta\le \frac{1}{2C_T}$, $B_{2\delta}(p)$ is locally contractible, any loop lying in $B_{\delta}(p)$ at $p$ is homotopic to a joining of loops not longer than $3\delta$ at $p$. By a standard argument of Gromov's short basis, the generating set of $\Gamma_p(p;\delta)$ can be chosen to have at most $k(n)$ elements. By Lemma \ref{lem-packing-index}, let $\delta(n)=\frac{1}{300C_TL(\epsilon(n),n)}$, then for any $0<\delta\le \delta(n)$, $[\Gamma_p(p;\delta):\Gamma_p(p;\delta)\cap H]\le c(n)$.
	
	Therefore, $\Gamma_p(p;\delta)\cap H$ is a subgroup of $H$, which is $w(n)$-nilpotent and has finite index $\le c(n)$ in $\Gamma_p(p;\delta)$, we derive that $\Gamma_p(p;\delta)$ itself is $w'(n)$-nilpotent.
\end{proof}

What remains in this paper is devoted to prove Theorem \ref{thm-good-point}. We will argue by contradiction. Assuming the contrary, then there is a sequence $(X_\alpha,p_\alpha)$ of Alexandrov $n$-spaces with curv $\ge -1$, such that for any $q_\alpha\in B_{\frac{1}{2C_T}}(p_\alpha)$, $\Gamma_{p_\alpha}(q_\alpha;\alpha^{-1})$ fails to be $w(n)$-nilpotent.

By passing to a subsequence, we may assume that $(X_\alpha, p_\alpha)\overset{GH}{\longrightarrow}(X_\infty,p_\infty)$, i.e., $(X_\alpha, p_\alpha)$ Gromov-Hausdorff converges to a limit space $(X_\infty,p_\infty)$.
Following \cite{KPT10}, we will show in the remaining sections that:

\begin{claim}\label{claim-good-point}
By passing to a subsequence, there is $0<R_1\le \frac{1}{16C_T}$ such that for each sufficient large $\alpha$, there is a point $q_\alpha\in B_{\frac{1}{2C_T}}(p_\alpha)$ and a chain of subgroups 
$$G_{1,\alpha}=\Gamma_{p_\alpha}(q_\alpha;R_1)\vartriangleright G_{2,\alpha} \vartriangleright\cdots \vartriangleright G_{l,\alpha}=\{e\} \qquad (l\le n)$$
satisfying
\begin{enumerate}
	\item[(A)] $G_{i,\alpha}/G_{i+1,\alpha}$ is $C$-abelian;
	\item[(B)] $G_{i,\alpha}\vartriangleleft G_{1,\alpha}$;
	\item[(C)] By (A) and (B), $G_{1,\alpha}$ acts on $G_{i,\alpha}/G_{i+1,\alpha}$ by conjugation, which induces a homomorphism $\rho_{i,\alpha}:G_{1,\alpha}\to \operatorname{Out}(G_{i,\alpha}/G_{i+1,\alpha})$. The image $\rho_{i,\alpha}(G_{1,\alpha})$ has finite elements, $\#\rho_{i,\alpha}(G_{1,\alpha})< N_0$.
\end{enumerate}	
\end{claim}

Then by \cite[Lemma 4.2.1]{KPT10}, we derive that $G_{1,\alpha}$ is $w(C, N_0)$-nilpotent, a contradiction.

In order to construct each $G_{i,\alpha}$, we define the local fundamental groups (Definition \ref{def-leveled-gap} below). Then (B) and (C) would follow from the leveled gap property (Definition \ref{subsec-local-group}) and a universal estimate of gradient push associated to a $\delta^2$-maximal frame (Definition \ref{def-maximal-frame}); see Section \ref{sec-B-C-prime}. (A) will be guaranteed by the construction and the generalized Bieberbach theorem (\cite{FY92}, cf. \cite{Ya96}); see Proposition \ref{prop-existence-local-group}.

\subsection{Local fundamental group and numerical maximal frame}\label{subsec-local-group}

We first introduce the local fundamental group that will realize $G_{i,\alpha}$.

\begin{definition}\label{def-local-group}
	Let $X$ be a locally complete Alexandrov $n$-space with curv $\ge -1$. Let $p$ be a point in $X$ such that the metric ball $B_1(p)$ is relative compact in $X$. For $0<r\le \frac{1}{2}$, the \emph{$r$-local fundamental group} $\pi_1^L(p;r)$ at $p$ is defined to be
	$$\pi_1^L(p;r)=\left< \text{loop $\gamma$ at $p$}:\operatorname{im}\gamma \in B_{r}(p)\right>/\sim,$$
	where $\gamma_1\sim \gamma_2$ if they are homotopic in $B_{2r}(p)$.	
\end{definition}
For $r_1>r_2>0$, let $\imath:\pi_1^L(p;r_2)\to \pi_1^L(p;r_1)$ be the inclusion homomorphism. A key property used in proving (B) and (C) is certain ``leveled gap'' between local fundamental groups at different scales as follow.
\begin{definition}\label{def-leveled-gap}
	We say that $\pi_1^L(p;R_1)$ satisfies \emph{$(\epsilon,\sigma,l)$-leveled gap property}, if there is a sequence of intervals $[r_l=0,R_l]$, $\cdots$, $[r_1,R_1]$ such that
	\begin{enumerate}\numberwithin{enumi}{theorem}
		\item $r_{i}\le \epsilon R_i$, and $r_i/R_{i+1}\le \sigma$,
		\item $\imath:\pi_1^L(p;r_i)\to \pi_1^L(p;R_i)$ is an isomorphism,
		\item $\imath(\pi_1^L(p;R_{i+1}))\vartriangleleft  \pi_1^L(p;r_i)$.
	\end{enumerate}		
\end{definition}

In practice, $r_i=3\diam Y_{i+1}$, where $Y_{i+1}$ is a ``regular fiber'' at $i$-level, which by definition, is a level set of $F_{k_1+\cdots+k_i}=(\dist_{a_1},\cdots,\dist_{a_{k_1+\cdots+k_i}})$, where $a_j$ are from a maximal $(k_1+\cdots+k_i)$-frame (for definition see below), $\dist_{a_j}$ is the distance function to $a_j$, and $2R_i$ is the radius of a Perelman's fibration $F_{k_1+\cdots+k_i}$'s base disk around a regular point in a limit space.

Secondly, we introduce a $\delta^2$-maximal frame. Let $X$ be an Alexandrov $n$-space and let $k$ be a positive integer $\le n$. Let $0<\delta\le \frac{1}{10^2}$. By \cite{BGP92}, given a pair of points $(a_1,b_1)$ and a minimal geodesic segment $[a_1b_1]$ between them, a $k$-frame $\{[a_ib_i]\}_{i=1}^k$, which consists of $k$ minimal geodesic segment $[a_ib_i]$, can be built up successively (and non-uniquely) on $X$: Assuming $[a_{i-1}b_{i-1}]$ is well-defined, then take $[a_ib_i]$ on $X$ that satisfies the following
\numberwithin{enumi}{theorem}
\addtocounter{theorem}{1}
\begin{enumerate}
	\item $b_i$ is the middle point of the geodesic segment $[a_{i-1}b_{i-1}]$,
	\item $|a_ia_{j}|=|b_ia_{j}|$, for all $1\le j\le i-1$,
	\item\label{delta-col} the edge $[a_ib_i]$ is \emph{$\delta^2$-collapsed}, i.e., $|a_{i}b_{i}|\le \delta^2|a_{i-1}b_{i-1}|$.
\end{enumerate}

A little more generally, we will consider $k$-frames where $b_i$ is not far away from the middle point $m_{i-1}$ of $[a_{i-1}b_{i-1}]$. Let $\{[a_ib_i]\}_{i=1}^k$ be a $k$-frame. Let $F_k=(\dist_{a_1},\cdots,\dist_{a_k}):X\to \mathbb R^k$.
A new pair $(a_{k+1},b_{k+1})$ is called \emph{$\delta^2$-maximal} relative to a $k$-frame $\{[a_ib_i]\}_{i=1}^k$ if 
\addtocounter{theorem}{1}
\begin{enumerate}
	\item $b_{k+1}$ is $\frac{\delta}{100}|a_kb_k|$-close to the middle point $m_k$ of $[a_kb_k]$,
	\item $F_{k}(a_{k+1})=F_k(b_{k+1})$,
	\item\label{cond-num-maximal} $|a_{k+1}b_{k+1}|= d_{k+1}$, where $$d_{k+1}=\max\{|b_{k+1}x|: x\in F_k^{-1}(F_k(b_{k+1})) \text{ and } |x b_{k+1}|\le  \delta^2\min_{i=1,\dots, k}|a_{i}b_{i}|\}.$$
\end{enumerate} 
Note that by (\ref{cond-num-maximal}), one always has $|a_{k+1}b_{k+1}|\le  \delta^2\min_{i=1,\dots, k}|a_{i}b_{i}|$.

\begin{definition}\label{def-maximal-frame}
	A $k$-frame is called $\delta^2$-\emph{maximal} if for each $2\le i\le k$, $(a_i,b_i)$ is $\delta^2$-maximal relative to the $(i-1)$-frame $\{[a_{j}b_{j}]\}_{j=1}^{i-1}$. For an $\delta^2$-maximal $n$-frame $\{[a_{j}b_{j}]\}_{j=1}^{n}$, we say that $\{[a_{j}b_{j}]\}_{j=1}^{n}$ is \emph{centered} at $x\in X$, if the point $x$ is $\frac{\delta}{100}|a_nb_n|$-close to $m_n$.
\end{definition}

By the construction above, Theorem \ref{thm-gradientpush-1} is reduced to a universal estimate of gradient push associated to a $\delta^2$-maximal frame; see Theorem \ref{thm-gradientpush} in the appendix.

The following fact on the gradient flow of $\lambda$-concave functions on Alexandrov space is applied in proving Theorem \ref{thm-good-point}.
\begin{theorem}[\cite{Petr07}]\label{thm-gradient-lip}
	Let $\Phi_t$ be the gradient flow of a $\lambda$-concave function on a complete Alexandrov space. Then $\Phi_t:X\to X$ is $e^{\lambda t}$-Lipschitz.
\end{theorem}

\begin{remark}
	We remark that all results on gradient push with respect to a $\delta^2$-maximal frame also hold for  a $(n,\delta)$-strainer with suitable maximum property. We only use maximal frames in this paper for simplicity.
\end{remark}
\section{Proofs of Claims (B) and (C)}\label{sec-B-C-prime}
We now prove that the existence of $(\epsilon,\sigma,l)$-leveled gap property and a $\delta^2$-maximal frame centered at $p$ would implies (B) and (C) hold for $G_i=\imath\pi_1^L(q;R_i)$ with $q\in B_\frac{1}{2C_T}(p)$.

Throughout this subsection, we always assume that $X$ is a locally complete Alexandrov $n$-space with curv $\ge -1$ such that the metric ball $B_1(p)$ is relative compact in $X$.

Let $q\in B_\frac{1}{2C_T}(p)$, $0<R_1\le \frac{1}{16C_T}$. Let $\pi_1^L(q;R_i)$ be a local fundamental group satisfying  the $(\epsilon,\sigma,l)$-leveled gap property. 
Let $G_i=\imath\pi_1^L(q;R_i)$ for each $1\le i\le l$. Then by the proofs in \cite{KPT10}, (B) and (C) hold for $G_i$. We give a proof for completeness.

\begin{proposition}[\cite{KPT10}]\label{prop-B-C-for-local-group}
	For $0<R_1\le \frac{1}{16C_T}$ and $\sigma>0$, there is $\epsilon(n)>0$ such that for $0<\epsilon<\epsilon(n)$, any local fundamental group $\pi_1^L(q;R_i)$ with $(\epsilon,\sigma,l)$-leveled gap property for intervals $[r_l=0,R_l]$, $\cdots$, $[r_1,R_1]$, if there is a $\delta^2$-maximal frame $\{[a_jb_j]\}_{j=1}^n$ centered at $q$ such that $$|a_1b_1|=\min \{\frac{1}{2C_T},\frac{1}{2}\diam X\},$$ then the chain of groups $G_i=\imath\pi_1^L(q;R_i)$, namely
	$$G_1\vartriangleright G_2\vartriangleright\cdots \vartriangleright G_l=\{e\},$$
	satisfies (B) and (C). 
\end{proposition}

Let $S_0$ be a short basis of $\pi_1(B_1(p),q)$ and $S_i=(S_0\cap G_{i})\cup (S_0\cap G_i)^{-1}$. For any $\gamma\in G_1$, the norm $|\gamma|$ is defined to be is the minimal length of its representative loops. The following elementary fact will be used in proving Lemma \ref{lem-conj-lip-control} below and (B), (C).

\begin{lemma}\label{lem-keep-gap}
	Any element $\gamma\in S_i\setminus S_{i+1}$ has norm 
	$$2R_{i+1}\le|\gamma|\le \frac{2\epsilon}{3}\cdot \min_{\beta\in S_{i-1}\setminus S_i}|\beta|$$
	and $G_i=\left< S_i\right>$.
\end{lemma}
\begin{proof}
	Since $\imath:\pi_1^L(q;r_i)\to \pi_1^L(q;R_i)$ is an isomorphism, any loop lying in $B_{R_i}(q)$ at $q$ is homotopic to a loop lying in $B_{r_i}(q)$ at $q$. Furthermore, since $B_{2r_i}(q)$ is locally contractible, any loop lying in $B_{r_i}(q)$ at $q$ is homotopic to a joining of loops not longer than $3r_i$ at $q$. 
	
	Because $S_0$ is a short basis of $\pi_1(B_1(p),q)$, it can bee seen that for any $\gamma\in S_i\setminus S_{i+1}$, 
	$$2R_{i+1}\le |\gamma|\le 3r_{i}$$
	and $G_i=\left< S_i\right>$.
\end{proof}

Via gradient push by a $\delta^2$-maximal $n$-frame on certain cover $\hat X$ of $B_1(p)$ and a $\delta^2$-maximal $n$-frame centered at $q$, up to a conjugation any loop in $G_1$, whose action on $\hat X$ has a definite displacement, admits the following control in Lemma \ref{lem-conj-lip-control}, which is essential in proving (B) and (C).

\begin{lemma}[\cite{KPT10}]\label{lem-conj-lip-control}
	Assume that there is a $\delta^2$-maximal frame $\{[a_jb_j]\}_{j=1}^n$ centered at $q$ such that $|a_1b_1|=\min \{\frac{1}{2C_T},\frac{1}{2}\diam X\}$.
	Suppose that $G_i\vartriangleleft G_1$. Then for any element $\gamma=\gamma_1*\cdots*\gamma_m$ with $\gamma_j\in S_1$ with $|\gamma| \le \frac{R_1}{100C(n)}$, there is $\beta\in G_{i}$ such that for any loop $\alpha\in G_{i}$ with $|\alpha|\le 3r_i$,
	$$|(\gamma*\beta)^{-1}*\alpha*(\gamma*\beta)|\le e^{2 \frac{\cosh2}{\sinh2} (2T(n)+C(n))} |\alpha|,$$
	where $C(n)$ is the constant in Remark \ref{rem-push-length}, $T(n)$ is the constant in Theorem \ref{thm-gradientpush}, and $(\hat X_i,\hat q_i) \overset{\pi_i}{\to} (B_1(p),q)$ is a suitable defined cover with $\pi_{i*}\pi_1(\hat X_i,\hat q_i)=G_{i}$.
\end{lemma}
\begin{proof}
	Up to a lifting to a cover $(\hat X_1,\hat q_1) \overset{\pi_1'}{\to} (B_1(p),q)$ with $$\pi_{1*}'\pi_1(\hat X_1,\hat q_1)=G_{1},$$ we assume that $\pi_1(B_1(p),q)=G_1$. Indeed, by the definition of $G_1$, $\pi_1'$ maps $B_{R_1}(\hat q_1)$ homeomorphically onto $B_{R_1}(q)$. If we want to construct a homotopy lying in $B_{R_1}(\hat q_1)$ of a short loop, we can actually do the construction in $X$ with the resulting homopoty lies in $B_{R_1}(q)$, then composite this homotopy by $(\pi_1'|_{B_{R_1}(\hat q_1)})^{-1}$.
	
	Let $(\hat X_i,\hat q_i) \overset{\pi_i}{\to} (B_1(p),q)$ be a cover with $\pi_{i*}\pi_1(\hat X_i,\hat q_i)=G_{i}$. Then by our assumption, $\pi_i$ is a normal cover. (This assumption will also be used in proving (B) and (C).)
	
	Let us construct a $\delta^2$-maximal frame $\{[\hat c_j\hat o_j]\}_{j=1}^n$ on $\hat X_i$ such that $\pi_i(\hat c_1)=q$ and $|\hat c_1\hat o_1|=\min \{\frac{R_1}{100C(n)},\frac{1}{2}\diam \hat X_i\}$. Let $\hat q_i'$ be a regular centered point of $\{[\hat c_j\hat o_j]\}_{j=1}^n$, i.e., $\hat q_i'$ is close to the middle point $\hat m_n$ of $[\hat c_n\hat o_n]$.
	
	Since $|q\pi_i(\hat q_i')|\le \min \{\frac{R_1}{100C(n)},\diam X\}$, there is a gradient push $\varphi$ of $\{[a_jb_j]\}_{j=1}^n$ in time $\le T(n)$ such that $\varphi(q)=\pi_i(\hat q_i')$, which gives rise to a homotopy $H$ from $\alpha$ to a loop $\varphi\circ \alpha$ at $\pi_i(\hat q_i')$. Moreover, the whole pushing line of broken geodesics has total length $\le C(n)\cdot |q\pi_i(\hat q_i')|$ (see Remark \ref{rem-push-length}).
	
	Since $G_{i}\vartriangleleft G_1$, there exists a lifting $\hat \alpha$ of $\alpha$ at $\gamma\hat q_i$, and a lifting homotopy $\hat H$ of $H$ on $\hat X_i$ from $\hat \alpha$ to $\hat\alpha'=\widehat{\varphi\circ \alpha}$, whose base points are $\gamma\hat q_i$ and $\hat q_i''$.  Then $\hat H$ and $\hat \alpha'$ lie in $B_{R_1}(\gamma\hat q_i)$.
	
	Moreover, there exists a deck transformation $\psi$ that maps $\hat q_i''$ to $\hat q_i'$. Let $\{\psi^{-1}[\hat c_j\hat o_j]\}$ be the pullback $\delta^2$-frame at $\hat q_i''$. Then there is a gradient push $\hat\phi$ of $\{\psi^{-1}[\hat c_j\hat o_j]\}$ in time $\le T(n)+C(n)$, which gives rise to a homotopy from $\hat\alpha'$ to $\hat \alpha''$, whose base point is $\hat q_i$.
	
	Joining two homotopies above together, we get a homotopy from $\hat \alpha$ to $\hat \alpha''$, whose base points are $\gamma \hat q_i$ and $\hat q_i$ respectively.
	
    Note that any single step in these two homotopies are defined by a gradient flow of $\frac{1}{2}\dist_x^2$ with $\dist_x<2$ for some $x$, hence the concavity of $\frac{1}{2}\dist_x^2$ is bounded by $2 \frac{\cosh2}{\sinh2}$. By Theorem \ref{thm-gradientpush} and Theorem \ref{thm-gradient-lip}, the length of $\pi(\hat \alpha'')$ satisfies
	$$\operatorname{length}\pi(\hat \alpha'')\le e^{2 \frac{\cosh2}{\sinh2} (2T(n)+C(n))}\cdot 
	\operatorname{length}\alpha.$$
	Let $\gamma'$ be the successive joining of push curves of $\varphi$ and $\pi\hat\phi$. Then it is clear that $\alpha''$ is homotopic to $\gamma'^{-1}*\alpha*\gamma'$, and there is $\beta\in G_i$ such that $\gamma'=\gamma*\beta$. 
\end{proof}

\vspace{2mm}
\begin{proof}[Proof of (B) in Proposition \ref{prop-B-C-for-local-group}] 
	~
	
	By definition of leveled gap property (Definition \ref{def-leveled-gap}), $G_{2}\vartriangleleft G_{1}$. We now prove $G_3\vartriangleleft G_1$.
	
	Let $(\hat X_2,\hat q_2) \overset{\pi_2}{\to} (B_1(p),q)$ be the normal cover defined in the proof of Lemma \ref{lem-conj-lip-control}.
	For any $\gamma\in S_1$, $\gamma$ satisfies that $|\gamma| \le \frac{R_1}{100C(n)}$ as $\epsilon$ in Definition \ref{def-leveled-gap} sufficient small. There is $\beta\in G_2$ such that $\gamma'=\gamma*\beta$, for any $\alpha\in S_3$, $$|\gamma'^{-1}*\alpha*\gamma'|\le e^{2 \frac{\cosh2}{\sinh2} (2T(n)+C(n))}|\alpha|.$$
	
	Let us take $\epsilon^{-1}>300C(n)e^{2 \frac{\cosh2}{\sinh2} (2T(n)+C(n))}$, then by Lemma \ref{lem-keep-gap}, $\gamma'^{-1}*\alpha*\gamma'\in G_3$. Since $G_3\vartriangleleft G_2$, $\gamma^{-1}*\alpha*\gamma\in G_3$. This implies $G_3\vartriangleleft G_1$.
	
	Repeating the argument above for loops in each $G_i$ for $i\ge 4$ successively, we complete the proof. 
\end{proof}

\vspace{2mm}
\begin{proof}[Proof of (C) in Proposition \ref{prop-B-C-for-local-group}]
	~
	
	For any fixed integer $m$, let $S_1^m=\{\gamma\in G_1: \operatorname{wordlength}(\gamma)\le m\}$. 
	
	Firstly, similar to the proof of (B), let $(\hat X_i,\hat q_i) \overset{\pi_i}{\to} (B_1(p),q)$ be the normal cover defined in the proof of Lemma \ref{lem-conj-lip-control}. For any $\gamma\in S_1^m$, $\gamma$ satisfies that $|\gamma| \le \frac{R_1}{100C(n)}$ as $\epsilon$ in Definition \ref{def-leveled-gap} sufficient small. There is $\beta\in G_i$ such that $\gamma'=\gamma*\beta$, for any $\alpha\in S_i$, $$|\gamma'^{-1}*\alpha*\gamma'|\le e^{2 \frac{\cosh2}{\sinh2} (2T(n)+C(n))}|\alpha|.$$
	
	Secondly, let us consider the normal cover $(\hat X_{i+1},\hat q_{i+1}) \overset{\pi_{i+1}}{\to} (B_1(p),q)$ defined in the proof of Lemma \ref{lem-conj-lip-control}. Then the relative volume comparison holds in $B_{R_1}(\hat q_{i+1})$ (see \cite{LiRong12}). By counting the lattice points $G_1(\hat q_{i+1})$ in balls of $(\hat X_{i+1},\hat q_{i+1})$, up to an inner automorphism of $G_i/G_{i+1}$ the possibility of transformations $\rho_i(S_1^m)$ on $G_i/G_{i+1}$ is bounded by the following number
	$$\left(\frac{\operatorname{vol}B_{-1}^n(e^{2 \frac{\cosh2}{\sinh2} (2T(n)+C(n))}\cdot 3r_{i+1}+R_{i+2})}{\operatorname{vol}B_{-1}^n(R_{i+2})}\right)^{\# (S_i\setminus S_{i+1})}.$$

	Let $N_0=\sup_{0<R_{i+2}\le 1}\left[\left(\frac{\operatorname{vol}B_{-1}^n((e^{2 \frac{\cosh2}{\sinh2} (2T(n)+C(n))}\cdot 3\sigma+1)R_{i+2})}{\operatorname{vol}B_{-1}^n(R_{i+2})}\right)^{c(n)}\right]+2$, where $c(n)$ is an upper bound of the total number of short basis $\#S_1$.
	
    Let us take $\epsilon^{-1}>300C(n)N_0$, then by Lemma \ref{lem-keep-gap}, $\#\rho_i(S_1^{N_0})< N_0$. Thus by \cite[Trivial Lemma 4.2.2]{KPT10}, $\#\rho_i(G_1)< N_0$.
\end{proof}

\section{Proof of Claim (A)}
To finish the proof of Theorem \ref{thm-good-point}, it suffices to construct the local fundamental groups and a maximal frames associated to a contradicting sequence $(X_\alpha, p_\alpha)\overset{GH}{\longrightarrow}(X_\infty,p_\infty)$, and then verify (A).

\begin{proposition}\label{prop-existence-local-group}
	Let $(X_\alpha,p_\alpha)\overset{GH}{\longrightarrow}(X_\infty,p_\infty)$ be a convergence sequence of Alexandrov $n$-spaces with curv $\ge -1$ such that $\diam X_\infty\ge 1$.
	Then by passing to a subsequence of $(X_\alpha, p_\alpha)$, there are $0<R_1\le \frac{1}{16C_T}$, $\sigma>0$, $1\le l\le n$, and for all large $\alpha\in \mathbb N$ there exist a point $q_\alpha\in B_{\frac{1}{2C_T}}(p_\alpha)$ such that
	\begin{enumerate}
		\item \label{item-construction} there is an associated $\delta^2$-maximal $n$-frame centered at $q_\alpha$ with $a_{1,\alpha}=p_\alpha$, $|a_{1,\alpha}b_{1,\alpha}|=\frac{1}{2C_T}$;
		\item \label{item-leveled-gap} the $R_1$-local fundamental group $\pi_1^L(q_\alpha; R_1)$ satisfies $(\epsilon_\alpha,\sigma,l)$-leveled gap property with respect to $[r_{l,\alpha}=0,R_{l,\alpha}]$, $\cdots$, $[r_{1,\alpha},R_{1,\alpha}=R_1]$ and $\epsilon_\alpha\to 0$;
		\item \label{item-almost-abelian} $\pi_1^L(q_\alpha;r_{i,\alpha})/\imath\pi_1^L(q_\alpha;r_{i+1,\alpha})$ is $C$-abelian for some constant $C$.
	\end{enumerate}
\end{proposition}

Now Theorem \ref{thm-good-point} follows from earlier arguments in Section \ref{sec-B-C-prime} and Proposition \ref{prop-existence-local-group}.

\vspace{2mm}
\begin{proof}[Proof of Theorem \ref{thm-good-point}]
	~
	
	Continue from earlier discussion, we have assumed a contradicting sequence $(X_\alpha,p_\alpha)$ of Alexandrov $n$-spaces with curv $\ge -1$, such that for any $q_\alpha\in B_{\frac{1}{2C_T}}(p_\alpha)$, $\Gamma_{p_\alpha}(q_\alpha;\alpha^{-1})$ fails to be $w(n)$-nilpotent, and $(X_\alpha, p_\alpha)$ Gromov-Hausdorff converges to a limit space $(X_\infty,p_\infty)$. Up to changing $X_\alpha$ to $X_\alpha \times \mathbb R^1$, we further assume that $\diam X_\infty\ge 1$.
	
	By Section \ref{sec-B-C-prime}, it suffices to construct a sequence of local groups at some $q_\alpha\in B_{\frac{1}{2C_T}}(p_\alpha)$ with leveled gap property such that (A) holds for $G_{i,\alpha}=\imath\pi_1^L(q_\alpha;R_{i,\alpha})$, and there is $\delta^2$-maximal frames at $q_{\alpha}$.
	
	Since the construction follows from Proposition \ref{prop-existence-local-group}, the proof of Theorem \ref{thm-good-point} is complete.
\end{proof}

What remains of the paper is proving Proposition \ref{prop-existence-local-group}. The following non-collapsing property of maximal frame will be used in its proof.

\begin{lemma}\label{lem-maximal-frame-non-collapse}
	Let $(X_\alpha^n,p_\alpha)\overset{GH}{\longrightarrow}(A^k,p_\infty)$ be a sequence of Alexandrov $n$-spaces that converges to an Alexandrov $k$-space $A^k$. Let $\{[a_{i,\alpha}b_{i,\alpha}]\}_{i=1}^k$ be a $\delta^2$-maximal frame in $X_\alpha^n$ with $a_{1,\alpha}=p_\alpha$, $0<d\le d_{1,\alpha}=|a_{1,\alpha}b_{1,\alpha}|\le 1$. By passing to a sequence, $\{[a_{i,\alpha}b_{i,\alpha}]\}_{i=1}^k$ converges to a frame $\{[a_{i,\infty}b_{i,\infty}]\}_{i=1}^k$ in $A^k$. Then $|a_{k,\infty}b_{k,\infty}|>0$.
\end{lemma}

\begin{proof}
	Argue by induction on $i$. Assume $d_{i,\infty}=|a_{i,\infty}b_{i,\infty}|>0$ for $i<k$, it suffices to show $d_{i+1,\infty}>0$. 
	
	Since $\dim A^k>i$, there is a point $a'_{i+1,\infty}$ in $A^k$ such that $0<|a'_{i+1,\infty}b_{i+1,\infty}|\le \frac{1}{4}\delta^2 d_{i,\infty}$ and $F_{i,\infty}(a'_{i+1,\infty})=F_{i,\infty}(b_{i+1,\infty})$, where $F_{i,\infty}=(\dist_{a_{1,\infty}},\cdots,\dist_{a_{i,\infty}})$.
	
	Take points $a'_{i+1,\alpha}\in X_\alpha^n$ which converges to $a'_{i+1,\infty}$. 
	By \cite[Theorem 5.4]{BGP92} (or see Theorem \ref{thm-co-Lip-part-dist-coord} below), 
	without loss of generality we assume that $F_{i,\alpha}$ is $\frac{1-2i\delta}{\sqrt{i}}$-open, i.e., $\frac{\sqrt{i}}{1-2i\delta}$-co-Lipschitz, on $B_{\delta d_{i,\alpha}}(b_{i+1,\alpha})$. Hence there is  $a''_{i+1,\alpha}$ in $F_{i,\alpha}^{-1}(b_{i+1,\alpha})$ such that $|a''_{i+1,\alpha}a'_{i+1,\alpha}|\to 0$ as $\alpha\to \infty$.
	Since $d_{i+1,\alpha}$ is maximal, we derive $$d_{i+1,\alpha}\ge |a''_{i+1,\alpha}b_{i+1,\alpha}|>\frac{1}{2}|a'_{i+1,\infty}b_{i+1,\infty}|.$$
\end{proof}

\vspace{2mm}
\begin{proof}[Proof of Proposition \ref{prop-existence-local-group}]
	~
	
	(\ref{item-construction}) The construction will be done inductively as follows.
	
	The Starting Step. Assume $\dim X_\infty=k_1$. Since $\diam(X_\infty)\ge 1$, we are able to directly construct a $\delta^2$-maximal $k_1$-frame $\{[a_{j,\alpha}b_{j,\alpha}]\}_{j=1}^{k_1}$ in $B_1(p_\alpha)$ such that $a_{1,\alpha}=p_{\alpha}$ and $b_{1,\alpha}$ is a point such that $d_{1,\alpha}=|a_{1,\alpha}b_{1,\alpha}|_{X_\alpha}=\frac{1}{2C_T}$.  By Lemma \ref{lem-maximal-frame-non-collapse}, it converges to a 
	$k_1$-frame $\{[a_jb_j]\}_{j=1}^{k_1}$ in $X_\infty$. Let $m_{k_1}$ be the middle point of $[a_{k_1}b_{k_1}]$.
	
	Let $p_1\in X_\infty$ be a regular point such that $|p_1m_{k_1}|\le \frac{1}{200}\delta |a_{k_1}b_{k_1}|$.
	By \cite[Theorem 5.4]{BGP92}, there is $0<R_1\le \frac{1}{16C_T}$ such that $B_{2R_1}(p_1)$ is bi-Lipschitz to an open ball $D^{k_1}$ in the Euclidean space $\mathbb R^{k_1}$ with bi-Lipschitz constant almost $1$. By Perelman's fibration theorem \cite{Pr93}, 
	the map $F_{k_1}^{-1}\circ F_{k_1,\alpha}$ is a locally trivial fibration, where $F_k=(\dist_{a_1},\cdots,\dist_{a_k})$ is the map associate to the maximal $k$-frame.
	
	Let us chose $p_{1,\alpha}\in X_{\alpha}\to p_1$. Clearly,  $\{[a_{j,\alpha}b_{j,\alpha}]\}_{j=1}^{k_1}$ is also a $\delta^2$-maximal $k_1$-frame at $p_{1,\alpha}$.
	
	We define $r_{1,\alpha}=3\diam_{X_\alpha} F_{k_1,\alpha}^{-1} (F_{k_1,\alpha}(p_{1,\alpha}))$, i.e., the extrinsic diameter of a reguler fiber of the Perelman's fibration.

	Step 1. Let $Y_{2,\alpha}=F_{k_1,\alpha}^{-1}(F_{k_1,\alpha}(p_{1,\alpha}))$, and let $\theta_{2,\alpha}=\frac{1}{3}r_{1,\alpha}=\diam_{X_{\alpha}} Y_{2,\alpha}$. Since $\theta_{2,\alpha}\to 0$, it is easy to see $Y_{2,\alpha}$ is connected.
	
	By passing to a subsequence, we assume the rescaled sequence
	$$\theta_{2,\alpha}^{-1}(X_\alpha,p_{1,\alpha})\to (A_2\times \mathbb R^{k_1},p'_2).$$
	Moreover, as subsets, $Y_{2,\alpha}$ converges to $A_2\times \{0\}$.
	
	Assume $\dim A_2=k_2$. Let $b_{k_1+1,\alpha}=p_{1,\alpha}$, and $a_{k_1+1,\alpha}\in Y_{2,\alpha}$ be the farthest point away from $p_{1,\alpha}$. Starting with $[a_{k_1+1,\alpha}b_{k_1+1,\alpha}]$, we construct a $\delta^2$-maximal frame $$\{[a_{k_1+j,\alpha}b_{k_1+j,\alpha}]\}_{j=1}^{k_2},$$
	which by the same argument as in the Starting Step, converges to a $k_2$-frame in $A_2\times \{0\}$,
	$$\{[a_{k_1+j}b_{k_1+j}]\}_{j=1}^{k_2}.$$
	
	Let $p_2$ in $A_2\times\{0\}$ such that $|p_2m_{k_1+k_2}|\le \frac{1}{200}\delta |a_{k_1+k_2}b_{k_1+k_2}|$. 
	
	We similarly define $p_{2,\alpha}\in X_\alpha\to p_{2}$,  $R_{2,\alpha}=\theta_{2,\alpha}\cdot R_2$, where there is a Perelman's fibration $F_{k_1+k_2}^{-1}\circ F_{k_1+k_2,\alpha}:\theta_{2,\alpha}^{-1}X_{\alpha}\to A_2\times \mathbb R^{k_1}$ over $B_{2R_2}(p_2)$, which is bi-Lipschitz to an open ball $D^{k_1+k_2}$ in the Euclidean space $\mathbb R^{k_1+k_2}$ with bi-Lipschitz constant almost $1$. (Note that every component in $F_{k_1}$ is a canonical Busemann function on $\mathbb R^{k_1}$.)
	
	Let $r_{2,\alpha}=3\diam_{X_{\alpha}} F_{k_1+k_2,\alpha}^{-1}(F_{k_1+k_2,\alpha}(p_{2,\alpha}))$.	
	
	Step 2. Do the same process as in Step 1 for $Y_{3,\alpha}=F_{k_1+k_2,\alpha}^{-1}(F_{k_1+k_2,\alpha}(p_{2,\alpha}))$, and $\theta_{3,\alpha}=\frac{1}{3}r_{2,\alpha}=\diam_{X_{\alpha}} Y_{3,\alpha}$. 
	
	Let us repeat the process in Step 2 until $k_1+\cdots+k_l=n$, then we have constructed a $\delta^2$-maximal $n$-frame 
	$$\{[a_{j,\alpha}b_{j,\alpha}]\}_{j=1}^{n},$$
	centered at $p_{l,\alpha}$.
	 
	Let $q_{\alpha}=p_{l,\alpha}$, then (\ref{item-construction}) is complete.
	
	(\ref{item-leveled-gap})
	By definition, each $\pi_1^L(p_{l,\alpha}; R_1)$  satisfies $(\epsilon_\alpha,\sigma,l)$-leveled gap property, where $\epsilon_\alpha=\frac{r_{i,\alpha}}{R_{i,\alpha}}\to 0$, $\sigma=\max\{3R_2^{-1}, \cdots, 3R_l^{-1}\}$.
	
	Indeed, in order to verify  (\ref{item-leveled-gap}), it suffices to show that
	\begin{equation}\label{normal-by-Hurewicz}
		\imath\pi_1^L(p_{l,\alpha};R_{i+1,\alpha})\vartriangleleft  \pi_1^L(p_{l,\alpha};r_{i,\alpha}) \tag{3.15}
	\end{equation}
	
	(\ref{normal-by-Hurewicz}) follows from the Hurewicz fibration Theorem \ref{thm:Lipschitz-submersion}.
	Indeed, let $D_{i,\alpha}=F_{k_1+\cdots+k_i,\alpha}^{-1}(D^{k_1+\cdots+k_i})$.  Then 
	$$\pi_1^L(p_{l,\alpha}; R_{i,\alpha})\cong\pi_1^L(p_{l,\alpha};r_{i,\alpha})\cong\pi_1(D_{i,\alpha},p_{l,\alpha}).$$
	By the choice of $p_{l,\alpha}$, $\theta_{i+1,\alpha}^{-1}(D_{i,\alpha},p_{l,\alpha})\overset{GH}{\longrightarrow} (A_{i+1}\times \mathbb R^{k_1+\cdots+k_i}, p_{i+1})$, where $p_{i+1}$ is a regular point in $A_{i+1}\times\{0\}$. Let $(\hat D_{i,\alpha},\hat p_{l,\alpha})\overset{\pi_{i,\alpha}}{\to} (D_{i,\alpha},p_{l,\alpha})$ be a cover of $D_{i,\alpha}$ with $$\pi_1(\hat D_{i,\alpha},\hat p_{l,\alpha})=\imath\pi_1^L(p_{l,\alpha};r_{i+1,\alpha}).$$
	We are to show that $\pi_{i,\alpha}$ is normal.
	
	Let $S_{i,\alpha}$ of be a short basis of $\pi_1^L(p_{l,\alpha};r_{i,\alpha})$. By passing to a subsequence, we assume that for the same $t_i$, $\{\gamma_{i,\alpha,1},\dots,\gamma_{i,\alpha,t_i}\}=S_{i,\alpha}\setminus \imath\pi_1^L(p_{l,\alpha};r_{i+1,\alpha}).$ 
	By the definition of a short basis, their lifting curve $\hat \gamma_{i,\alpha,1},\dots, \hat \gamma_{i,\alpha,t_i}$ are minimal geodesics in $\hat D_{i,\alpha}$ from $\hat p_{l,\alpha}$ to some $\hat q_{i+1,\alpha,1},\dots, \hat q_{i+1,\alpha,t_i}$ respectively. 
	
	It suffices to show that for any loop $\gamma\in \imath \pi_1^L(p_{l,\alpha};r_{i+1,\alpha})$, and any $\gamma_{i,\alpha,s}$,  there is a homotopy with fixed endpoint from $\gamma_{i,\alpha,s}*\gamma*\gamma^{-1}_{i,\alpha,s}$ to a loop in $\imath \pi_1^L(p_{l,\alpha};r_{i+1,\alpha})$.
	
	By passing to a subsequence, $\theta_{i+1,\alpha}^{-1}(\hat D_{i,\alpha},\hat p_{l,\alpha})\to  (\hat A_{i+1}\times \mathbb R^{k_1+\cdots+k_i}, \hat p_{i+1})$.
	And each minimal geodesic $\hat \gamma_{i,\alpha,s}$ converges to $\hat \gamma_{i,s}$ in $\hat A_{i+1}\times\{0\}$, which is a minimal geodesic form $\hat p_{i+1}$ to $\hat q_{i+1,s}$.
	
	If $\hat \gamma_{i,1},\cdots,\hat \gamma_{i,t_i}$ pass only regular points, then by \cite{BGP92} there is a positive $\eta>0$ such that the neighborhood $B_{2\eta}=U_{2\eta}(\bigcup_{s=1}^{t_i} \hat \gamma_{i,s})\subset \hat A_{i+1}\times \mathbb R^{k_1+\cdots +k_i}$ contains only $(k_1+\cdots +k_{i+1},\delta)$-strained points with a universal strainer radius. Thus by Theorem \ref{thm:Lipschitz-submersion} and Remark \ref{rem-local-fibration}, for $\alpha$ large we have
	\addtocounter{theorem}{1}
	\begin{enumerate}
		\item\label{homotopy-control-1} any $F_{k_1+\cdots+k_i,\alpha}$-fiber in $D_{i,\alpha}$ has extrinsic diameter not larger than $\eta/4$. 
		\item\label{homotopy-control-2} there is a Hurewicz fibration $\varphi_{i,\alpha}$ which is close to the original GHA, over $B_{2\eta}$, whose fiber's diameter $\le \eta/4$, such that $\hat \gamma_{i,\alpha,s}\subset \varphi_{i,\alpha}^{-1}(U_\eta(\hat\gamma_{i,s}))$ and $\varphi_{i,\alpha}(\hat \gamma_{i,\alpha,s})\subset U_{\eta/4}(\hat \gamma_{i,s})$.  
	\end{enumerate}
	
	Note that the lifting of $\gamma_{i,\alpha,s}*\gamma*\gamma_{i,\alpha,s}^{-1}$ at $\hat q_{i+1,\alpha,s}$ is $\hat \gamma_{i,\alpha,s}*\hat \gamma*\hat \gamma_{i,\alpha,s}^{-1}$ with $\hat \gamma$ a closed lifting of $\gamma$ at $\hat p_{l,\alpha}$. Then by the construction of $\varphi_{i,\alpha}$ (see Proposition \ref{prop:gradient-retract}), there is a canonical contraction from tubular neighborhood of a $\varphi_{i,\alpha}$-fiber to itself.
	Thus, by (\ref{homotopy-control-1})
	there is a homotopy $\hat H_1'$ maps $\hat \gamma $ to $\hat \gamma'\subset \varphi_{i,\alpha}^{-1}(\varphi_{i,\alpha}(\hat p_{l,\alpha}))$ keeping $\hat p_{l,\alpha}$ unmoved. Thus, we have a homotopy $\hat H_1$ maps $\hat \gamma_{i,\alpha,s}*\hat \gamma*\hat \gamma_{i,\alpha,s}^{-1}$ to $\hat \gamma_{i,\alpha,s}*\hat \gamma'*\hat \gamma_{i,\alpha,s}^{-1}$, keeping $\hat \gamma_{i,\alpha,s}$ and $\hat\gamma_{i,\alpha,s}^{-1}$ unmoved.
	
	Furthermore, by (\ref{homotopy-control-2})
	there is a homotopy $\hat H_2'$ maps $\hat \gamma'$ to $\hat \gamma''$, which lies in $\varphi_{i,\alpha}^{-1}(\varphi_{i,\alpha}(\hat q_{i+1,\alpha,s}))$, moving $\hat p_{l,\alpha}$ to $\hat q_{i+1,\alpha,s}$ along $\hat \gamma_{i,\alpha,s}$ such that $\gamma''$ is a loop at $\hat q_{i+1,\alpha,s}$. Thus, we have a homotopy $\hat H_2$ maps $\hat \gamma_{i,\alpha,s}*\hat \gamma'*\hat \gamma_{i,\alpha,s}^{-1}$ to $\hat \gamma''$, keeping $\hat q_{i+1,\alpha,s}$ unmoved.
	
	Then $\pi_{i,\alpha}(\hat H_2*\hat H_1)$ is a homotopy maps $\gamma_{i,\alpha,s}*\gamma*\gamma_{i,\alpha,s}^{-1}$ to $\pi_{i,\alpha}(\hat \gamma'')$ keeping $p_{l,\alpha}$ unmoved, and $\pi_{i,\alpha}(\hat\gamma'')$ lies in $\imath\pi_1^L(p_{l,\alpha};r_{i+1,\alpha})$.
	
	In order to complete the proof of (\ref{normal-by-Hurewicz}), we now verify that all limit minimal geodesics $\hat \gamma_{i,s}$ pass regular points.
	Firstly, it is clear that the limit projection $\pi_i:\hat A_{i+1}\times\{0\}\to A_{i+1}\times\{0\}$ is a submetry (i.e., $1$-LcL). Secondly, there is a neighborhood of $\hat p_{i+1}$ restricted on which $\pi_i$ is an isometry. This is because near $\hat p_{l,\alpha}$, there is a homeomorphic lifting of $D_{i,\alpha}$ in $\hat D_{i,\alpha}$. Hence $\dim \hat A_{i+1}=\dim A_{i+1}$, and all lift points $\hat p_{i+1}$, $\hat q_{i+1,1},\dots,\hat q_{i+1,t_i}$ are regular. By \cite{BGP92}, any minimal geodesic $\hat \gamma_{i,s}$ between them contains only regular points. 
	
	The proof of (\ref{item-leveled-gap}) is now complete.
	
	(\ref{item-almost-abelian}) 
	By (\ref{normal-by-Hurewicz}), $(\hat D_{i,\alpha},\hat p_{l,\alpha})\overset{\pi_{i,\alpha}}{\to} (D_{i,\alpha},p_{l,\alpha})$ is a normal cover, whose deck-transformation group is $\Lambda_{i,\alpha}=\pi_1^L(p_{l,\alpha};r_{i,\alpha})/\imath\pi_1^L(p_{l,\alpha};r_{i+1,\alpha})$.
	
	Since $\Lambda_{i,\alpha}$ equivariantly converges, the limit group $\Lambda_{i}$ acts on $\hat A_{i+1}\times \{0\}$ isometrically. By the generalized Bieberbach theorem \cite{FY92} (cf. \cite{Ya96}), $\Lambda_{i}$ is $C$-abelian.
	Since $\Lambda_i$ is a discrete group,
	the GHA $\rho_{i,\alpha}$ between $\Lambda_{i,\alpha}$ and $\Lambda_i$ is a homomorphism. 
	
	We now prove that $\rho_{i,\alpha}$ is an isomorphism.
	Firstly, since there is no non-trivial element of $\Lambda_{i,\alpha}$ whose displacement is shorter than $2R_{i+1,\alpha}$, $\rho_{i,\alpha}$'s kernel should be a subgroup $K_{i,\alpha}$, which moves $\hat p_{l,\alpha}$ to infinity. Secondly, since $\pi_1^L(p_{l,\alpha};r_{i,\alpha})$ is generated by all of its elements whose displacements are not longer than $3r_{i,\alpha}$ and any generating relation can be written as a word in these elements with wordlength $\le 3$, the corresponding property holds for $\Lambda_{i,\alpha}$. Because the relative volume comparison (see \cite{LiRong12}) provides an uniform bound to the number of $\Lambda_{i,\alpha}$-orbit points in $B_{9r_{i,\alpha}}(\hat p_{l,\alpha})$, by passing to a subsequence, the presentation of $\Lambda_{i,\alpha}$ is stable. Hence $\rho_{i,\alpha}$ is an isomorphism. 
\end{proof}

\section{Appendix on gradient push}

Let $\{[a_jb_j]\}_{j=1}^n$ be a $\delta^2$-maximal $n$-frame with $|a_1b_1|\le 1$ on an Alexandrov $n$-space $X$ with curv $\ge -1$. Let $F_k(x)=(\dist_{a_1}(x),\dist_{a_2}(x),\cdots,\dist_{a_k}(x)).$
Recall that by the definition of maximal frame, $d_j=|a_jb_j|$ which satisfies
\begin{equation}\label{def-key-max}
d_{k+1}=\max_{F_k^{-1}(F_k(b_{k+1}))}\{\min\{\dist_{b_{k+1}},  \delta^2\cdot\min_{i=1,\dots, k}|a_{i}b_{i}|\}\}.\tag{4.1}
\end{equation}
In the following we always assume that $m_k$ is the middle point of $[a_kb_k]$, and $b_{n+1}$ is a point $\frac{\delta}{100}|a_{n}b_{n}|$-close to the middle point $m_n$ of $[a_nb_n]$.

We restate Theorem \ref{thm-gradientpush-1} and give a proof in the following form.

\begin{theorem}[\cite{KPT10}]\label{thm-gradientpush}
	There is $T(n)>0$ such that for any $\delta^2$-maximal $n$-frame $\{[a_jb_j]\}_{j=1}^n$ with $|a_1b_1|\le 1$ on an Alexandrov $n$-space $X$ with curvature $\ge -1$, any point $b_{n+1}$ that is $\frac{\delta}{100}|a_nb_n|$-close to the middle point $m_n$ of $[a_nb_n]$ can be pushed successively by the gradient flows of $\frac{1}{2}\dist_{b_{n+1}}^2$, $\frac{1}{2}\dist_{a_j}^2$, $\frac{1}{2}\dist_{b_j}^2$ $(j=1,\dots, n)$ to any point $p\in B_{100 |a_1b_1|}(b_{n+1})$ in total time $\le T(n)$.
\end{theorem}
\begin{remark}\label{rem-push-length}
	By the proof of Theorem \ref{thm-gradientpush} (or by replacing $\frac{1}{2}\dist^2$ with $\dist$ in Theorem \ref{thm-gradientpush}), the length of broken gradient curves between $b_{n+1}$ and $p$ is no more than $C(n)\cdot |b_{n+1}p|$, where $C(n)=2n+\frac{4(n-1)}{\sin\sigma(n)}+1$ with $\sigma(n)$ in Lemma \ref{lem-angle-est-0}. This fact is also used in the proof of Theorem \ref{thmb-margulis}; see Lemma \ref{lem-conj-lip-control}. 
\end{remark}

Some partial motivations to write a proof other than just referring to \cite[Lemma 2.5.1]{KPT10} are as follows. 

Firstly, there is only a sketched proof for Theorem \ref{thm-gradientpush} in \cite{KPT10}, where the ratio bound on the pushing time $\frac{t_{k-1}}{t_k}\le \frac{1}{\delta^n}$ from level $d_k$ to $d_{k-1}$ is claimed without explanation in the proof of \cite[Lemma 2.5.1]{KPT10}. 

Since it is hard for us to follow at that point, we write a detailed proof on the surjectivity and universal speed of gradient pushing-out (using the maximum property (\ref{def-key-max}) and $\frac{1-2n\delta}{\sqrt{n}}$-openness in \cite[Theorem 5.4]{BGP92}). In particular, our proof leads to a sharpened universal time bound $n^2\delta^{-1}$, improving the universal time bound $\delta^{-n^2}$ claimed in \cite{KPT10}.
 
Secondly, a crucial difference between an Alexandrov space $X$ with curvature bounded below and a Riemannian manifold $M$ is that, there may be proper extremal subsets in $X$ such that no gradient curves can get out of them. Without further explanation, it is also hard for us to see from \cite{KPT10} that the gradient pushing-out process can be chosen to avoid extremal subsets.  

We fill more details and construct a specific gradient pushing broken line, consisting of $k$-regular (i.e., the tangent cone $T_pX$ at least splits off $\mathbb R^k$) or $(n,\delta)$-strained points when $a_j,b_j$ and the ending point $p$ are $k$-regular. 

Since all our estimates will hold for a new $n$-frame $\{[a_j'b_j']\}_{j=1}^n$, where $a_j'$ and $b_j'$ are nearby regular points around $a_j$ and $b_j$. It follows that gradient push between regular points only passes through regular points.

This provides a detailed justification for the gradient push in proving the Margulis lemma on an Alexandrov space.

\subsection{Proof of Theorem \ref{thm-gradientpush}}
The proof of Theorem \ref{thm-gradientpush} can be divided into two steps. 

\textbf{Step 1}. Prove that in at most a definite time $T(\delta,n)$, $b_{n+1}$ can be pushed to any point in a ball $B_{\frac{\delta}{2}d_n}(m_n)$ whose radius is at a small but fixed relative scale, where $m_n$ is the middle point of $[a_nb_n]$. 

\begin{lemma}\label{lem-pushingforward-0}
	For $0<\delta<\delta(n)$ and any $q\in B_{\frac{\delta}{2}d_n}(m_n)$, $b_{n+1}$ can be pushed by an at most countably succession of
	the gradient flows of $\frac{1}{2}\dist_{a_j}^2, \frac{1}{2}\dist_{b_j}^2$ to $q$ in time $\le \frac{12n^2}{4n-1}\delta$.
\end{lemma}
Compared with the proof of \cite[Theorem 5.4]{BGP92}, Lemma \ref{lem-pushingforward-0} follows from certain reversing argument, which will be given at the end of the appendix. 

\textbf{Step 2}. Prove that $B_{\frac{\delta}{2}d_n}(m_n)$ can be pushed outside further. If $\frac{d_n}{d_1}$ admits a definite lower bound $\tau$, then one may push $B_{\frac{\delta}{2}d_n}(m_n)\setminus B_{\frac{\delta}{4}d_n}(m_n)$ onto $B_{100d_1}(m_n)$ by $\frac{1}{2}\dist_{m_n}^2$ just one more time taking no more than $T(\tau,n)$. However, $d_n$ may be far less than $d_1$, or even $d_{n-1}$.

To overcome this difficulty, we divided an $n$-frame into several levels. We say that a $\delta^2$-maximal $n$-frame $\{[a_jb_j]\}_{j=1}^n$ is of \emph{$(\frac{\delta^2}{100},l)$-leveling} if there is $1\le k_1<\cdots<k_l=n$ such that $[a_{k_{i-1}+1}b_{k_{i-1}+1}],\dots, [a_{k_{i}}b_{k_{i}}]$ lies in the same level in the sense that
$d_{j}>\frac{\delta^2}{100}d_{j-1}$ for any integer $k_{i-1}+1\le j\le k_i$, $1\le i\le l$ ($k_0=0$), and $[a_{k_i}b_{k_i}]$, $[a_{k_i+1}b_{k_i+1}]$ lie in different levels, i.e.,
$d_{k_i+1}\le \frac{\delta^2}{100}d_{k_i}$ for any $1\le i\le l$.

Inside each $i$-th level, it follows from elementary gradient estimate that $B_{\frac{\delta}{4}d_{k_i}}(b_{n+1})$ can be pushed by the center $b_{n+1}$, i.e., $\frac{1}{2}\dist_{b_{n+1}}^2$, onto $B_{100d_{k_{i-1}+1}}(b_{n+1})$ in time $\le \ln (\frac{400}{\delta}(\frac{100}{\delta^2})^{k_i-(k_{i-1}+1)})$.

In order to push $B_{100d_{k_{i-1}+1}}(b_{n+1})$ further outside onto a large leveled ball in a specific way, we need to prove the following lemma.
\begin{lemma}\label{lem-difflevel-push-0}
	If $d_{k+1}=|a_{k+1}b_{k+1}|\le \frac{\delta^2}{100}d_k$, then for any $p\in B_{\frac{\delta}{2}d_k}(b_{k+1})$, there is some point $q\in B_{50d_{k+1}}(b_{k+1})$ which can be pushed successively along finitely-broken geodesics, each of which is pointing to one of $\{a_j,b_j\}_{j=1}^k$, by the gradient flows of $\frac{1}{2}\dist_{a_j}^2, \frac{1}{2}\dist_{b_j}^2$ to $p$ in time $\le C(n)\delta$.
\end{lemma}

Note that in the case of Lemma \ref{lem-difflevel-push-0} for different level, we are using endpoints of long edges in the frame, which lie outside the small ball $B_{50d_{k+1}}(b_{k+1})$.

In the proof of Lemma \ref{lem-difflevel-push-0}, the core is  
the following angle estimate, which follows from the numerical maximum property (\ref{def-key-max}) of $\delta^2$-maximal frame.

\begin{lemma}[Angle Estimate]\label{lem-angle-est-0}
	There is $\delta(n)>0$ such that the following holds for $0<\delta<\delta(n)$. If  $d_{k+1}=|a_{k+1}b_{k+1}|\le \frac{\delta^2}{100}|a_kb_k|$, then for any $p\in B_{\frac{\delta}{2}d_k}(b_{k+1})\setminus B_{50d_{k+1}}(b_{k+1})$, there exists $e\in \{a_j,b_j\}_{j=1}^k$ such that $\measuredangle(p;e,b_{k+1})\le \frac{\pi}{2}-\sigma(n)$.
\end{lemma}
\begin{proof}	
	Argue by contradiction. For any $\delta>0$, there is a $\delta^2$-maximal $k$-frame such that the conclusion of Lemma \ref{lem-angle-est-0} fails. Then by Toponogov comparison (cf. \cite[Lemma 5.6]{BGP92}), there is $\sigma=\sigma(\delta)\to 0$ as $\delta\to 0$ such that for any $x\in [b_{k+1}p]$ and every $1\le i\le k$, $$\measuredangle(x;a_i,b_{k+1})\ge \frac{\pi}{2}-\sigma, \text{ and } \measuredangle(x;b_i,b_{k+1})\ge \frac{\pi}{2}-\sigma.$$
	Then as $\sigma$ sufficient small,
	\begin{equation}\label{eq-dist-image}
	||a_ix|-|a_ib_{k+1}||\le |xb_{k+1}|\cdot \sin \sigma.
	\end{equation}  
	By \cite[Theorem 5.4]{BGP92} (or see Theorem \ref{thm-co-Lip-part-dist-coord} below), for any $0\le \sigma\le \frac{1}{2k}$ the partial distance coordinates map associated to $k$-subframe $\{[a_jb_j]\}_{j=1}^k$,
	$$F_k:X\to \mathbb R^k, \quad F_k(x)=(|a_1x|,|a_2x|,\cdots,|a_kx|),$$
	is $\frac{1-2k\sigma}{\sqrt{k}}$-open, i.e., $\frac{\sqrt{k}}{1-2k\sigma}$-co-Lipschitz, on $B_{\delta d_k}(b_{k+1})$. Hence there is $x'\in F_k^{-1}(F_k(b_{k+1}))\cap B_{\delta d_k}(b_{k+1})$ such that the distance
	\begin{equation}\label{eq-dist-preimage}
	|xx'|\le \frac{\sqrt{k}}{1-2k\delta}\cdot |F_k(x)-F_k(x')|,
	\end{equation} which is, by (\ref{eq-dist-image}), far less than $|xb_{k+1}|$. Let $|xb_{k+1}|=50d_{k+1}$, then as $\delta=\delta(n)$ sufficient small, $$|x'b_{k+1}|>> d_{k+1},$$
	a contradiction to the choice of $(a_{k+1},b_{k+1})$ in (\ref{def-key-max}).
\end{proof}

We now prove Lemma \ref{lem-difflevel-push-0}.

\vspace{2mm}
\begin{proof}[Proof of Lemma \ref{lem-difflevel-push-0}]
	~
	
	Let $e=e(p)$ be one of $\{a_j,b_j\}_{j=1}^k$ provided by Lemma \ref{lem-angle-est-0}, and let us connect $p$ and $e$ by a minimal geodesic $[pe]$. By Toponogov comparison and Lemma \ref{lem-angle-est-0}, there is a universal $\Delta r$ determined by the $(-1)$-law of cosine such that
	for any $p'\in [pe]$ with $|pp'|\le \Delta r$, one has
	$$0<\frac{|b_{k+1}p|-|b_{k+1}p'|}{|pp'|}\le  \sin\sigma(n).$$  	
	
	If $p'$ can be chosen that $[pp']\cap B_{50d_{k+1}}(b_{k+1})\neq \emptyset$, then $x$ is one of the intersection point and the geodesic $[xp]$ is the gradient flow of $\frac{1}{2}\dist_{e}^2$.
	
	Otherwise, let $p'=p$ with $|pp'|=\Delta r$. By repeating the process above successively, we get a finitely-broken geodesic from $p$ to some point $q\in B_{50d_{k+1}}(b_{k+1})$, whose reverse realizes the geodesic flows from $q$ to $p$ by endpoints $\{a_j,b_j\}_{j=1}^k$.
	
	Because for each $p'$ above, $|p'e(p')|\ge \frac{1-\delta}{2}d_k-\frac{\delta}{100}d_k$, and the total length of the broken geodesic is bounded by $\frac{1}{\sin\sigma} \delta d_k$, this completes the proof of Lemma \ref{lem-difflevel-push-0}. 
\end{proof}

Now we are ready to prove Theorem \ref{thm-gradientpush}.
\vspace{2mm}
\begin{proof}[Proof of Theorem \ref{thm-gradientpush}]
	~
	
	Let us assume that the $n$-frame $\{[a_jb_j]\}_{j=1}^n$ admit a $(\frac{\delta^2}{100},l)$-leveling, $1\le k_1<\cdots<k_l=n$. Let $k_0=0$.
	
	By Lemma \ref{lem-pushingforward-0}, $b_{n+1}$ can be pushed onto $B_{\frac{\delta}{2}}(b_{n+1})$ in time $\le 16\delta\frac{4n^2}{4n-1}$.
	
	In each $i$-th level, $B_{\frac{\delta}{4}d_{k_i}}(b_{n+1})$ can be pushed by $\frac{1}{2}\dist_{b_{n+1}}^2$ onto $B_{100d_{k_{i-1}+1}}(b_{n+1})$ in time $\le \ln (\frac{400}{\delta}(\frac{100}{\delta^2})^{k_i-(k_{i-1}+1)})$.
	
	From $i$-th level to $(i-1)$-th level, note that for any $100d_{k_{i-1}+1}\le r\le \frac{\delta}{4}d_{k_{i-1}}$, $B_{\frac{r}{2}}(b_{k_{i-1}+1})\subset B_r(b_{n+1})\subset B_{2r}(b_{k_{i-1}+1})$. By Lemma \ref{lem-difflevel-push-0},  $B_{100d_{k_{i-1}+1}}(b_{n+1})$ can be pushed onto  $B_{\frac{\delta}{4}d_{k_{i-1}}}(b_{n+1})$ in time $\le C(n)\delta$.
	
	Since it finishes after $2l$-steps, the proof completes.
\end{proof}

\subsection{Tracing back process} For completeness we give a proof for the co-Lipschitzness of $F_k:X\to \mathbb R^k$, which has been used in proving Lemma \ref{lem-angle-est-0}. Lemma \ref{lem-pushingforward-0} also follows similarly. The idea of proof is just the same as that of \cite[Theorem 5.4]{BGP92}.

\begin{theorem}[{\cite[Theorem 5.4]{BGP92}}]\label{thm-co-Lip-part-dist-coord}
	There is $\sigma(k)>0$ such that the following holds for $0\le \sigma\le \sigma(k)$.
	
	Let $\{a_j,b_j\}_{j=1}^k$ be a $(k,\sigma)$-strainer at $x_0$ with radius $$r_k=\min\{|a_jx_0|,|b_jx_0|\}_{j=1}^k\le \max\{|a_jb_j|\}_{j=1}^k\le 1.$$ Let $F_k:X\to \mathbb R^k$, $F_k(x)=(|a_1x|,\cdots,|a_kx|)$, be the map associated to $\{a_j,b_j\}_{j=1}^k$ that forms a partial distance coordinates around $x_0$.
	
	Let $p=p_0$ be a point in $B_{\frac{\sigma}{20} d_k}(x_0)$ such that \begin{equation}\label{ineq-co-Lip-part-dist-coord}
	||a_jp|-|a_jx_0||\le \frac{1}{4k}|px_0| \qquad (j=1,\dots, k).
	\end{equation} 
	Then there is a (infinitely-)broken geodesic $[p_0p_1^1\cdots p_1^kp_2^1\cdots p_2^kp_3^1\cdots]$, contained in $B_{\frac{\sigma}{10} d_k}(x_0)$ such that
	the endpoint $p_l=p_l^k$ converges to a point $p'$ as $l\to \infty$, which satisfies
	\begin{equation}
	|pp'|\le \frac{4k+1}{3\sqrt{k}}\cdot |F_k(p)F_k(x_0)|,\qquad F_k(p')=F_k(x_0).
	\end{equation}
\end{theorem}

Let $\delta>0$ be a small number other than $\sigma$. Let $p$ be a point in  $B_{\delta d_k}(x_0)$. Let us first define its \emph{$l$-th round $k$-tracing back point} $p_l=p_l^k$ of $p$ towards $x_0$'s $F_k$-fiber inductively as follows. Here tracing back means moving along gradient curves of distance to $a_j$ or $b_j$ backwards.

Let $p_0=p$ and let us assume that $p_{l-1}$ is well-defined. For the first coordinate function $f_1=\dist_{a_1}$, let $p_{l}^1$ be a point lies in the broken geodesic $[a_1p_{l-1}b_1]$ such that $$f_1(p_{l}^1)-f_1(p_{l-1})=f_1(x_0)-f_1(p_{l-1}).$$ Let $p_l^2$ be a point lies in the broken geodesic $[a_2p_{l-1}^1b_2]$ such that
$$f_2(p_{l}^2)-f_2(p_{l-1}^1)=f_2(x_0)-f_2(p_{l-1}).$$
Repeating $k$-times, we have $p_{l-1}^k$ in $[a_kp_{l-1}^{k-1}b_k]$ such that
$$f_k(p_{l}^k)-f_k(p_{l-1}^{k-1})=f_k(x_0)-f_k(p_{l-1}).$$
Then $p_l$ is defined to be $p_l^k$.

Since $p\in B_{\delta d_k}(x_0)$, by an elementary angle estimate \cite[Lemma 5.6]{BGP92}, the following holds for $0<\delta<\frac{\sigma}{10}$:
for any $1\le i\le k$, $1\le j\le i-1$,
$$\begin{aligned}
|\measuredangle (p;a_{j},a_{i})-\frac{\pi}{2}|\le 4\sigma,\quad 
|\measuredangle (p;b_{j},a_{i})-\frac{\pi}{2}|\le 4\sigma,\\
|\measuredangle (p;a_{j},b_{i})-\frac{\pi}{2}|\le 4\sigma,\quad 
|\measuredangle (p;b_{j},b_{i})-\frac{\pi}{2}|\le 4\sigma.
\end{aligned}$$
Clearly, it follows that the relations below hold.
\begin{lemma} For some positive function $\epsilon=\epsilon(\sigma)\to 0$ as $\sigma\to 0$,
	\begin{enumerate}
		\item\label{tracingback-est1} $|p_{l}^ip_{l}^{i-1}|\le (1+\epsilon)\cdot |f_i(x_0)-f_i(p_l)|$;
		\item\label{tracingback-est2} $|f_j(p_l^i)-f_j(p_l^{i-1})|\le \epsilon \cdot  |f_i(x_0)-f_i(p_{l-1})|$ for any $j\neq i$.
	\end{enumerate}
	
\end{lemma}

Now we are ready to prove Theorem \ref{thm-co-Lip-part-dist-coord}.

\vspace{2mm}
\begin{proof}[Proof of Theorem \ref{thm-co-Lip-part-dist-coord}]
	~
	
	Let $\delta=\frac{\sigma}{20}$.
	Let $A_l=\sum_{j=1}^k|f_j(p_l)-f_j(x_0)|$ and $B_l=|p_{l+1}p_l|$. 
	As long as the $l$-th tracing back point $p_l$ lies in $B_{2\delta d_k}(x_0)$, the estimates (\ref{tracingback-est1})-(\ref{tracingback-est2}) hold. By triangle inequality, we derive
	$A_{l+1}\le \epsilon (k-1)A_l$ and $B_l\le (1+\epsilon)A_l$.
	As $\delta$ sufficient small, $\epsilon\le \frac{1}{4k}$ so that
	$A_{l+1}\le \frac{1}{4}A_l$ and $B_l\le \frac{4k+1}{4k}A_l$, and thus $A_l\le \frac{1}{4^l}A_0$ and $B_l\le \frac{4k+1}{4k}\cdot\frac{1}{4^l}A_0$ are Cauchy sequences. 
	
	Now let us check that, by induction on $l$,  each $p_l$ satisfies $|p_lx_0|\le \frac{3}{2}\delta d_k$ so that $p_l\in B_{2\delta d_k}(x_0)$.  By the assumption (\ref{ineq-co-Lip-part-dist-coord}), 	
	$A_0\le \frac{1}{4}|px_0|$, and thus $A_l\le \frac{1}{4^{l+1}} |px_0|$, $B_l\le \frac{4k+1}{4k}\cdot \frac{1}{4^{l+1}} |px_0|$. Then 
	$$\sum_{t=0}^{l}B_t\le \frac{4k+1}{4k}\cdot \frac{1}{3} |px_0| \le \frac{1}{2}|px_0|,$$
	which justifies $|p_lx_0|\le \frac{3}{2}\delta d_k$. 
	
	Let $p'$ be the limit point of $p_l$, then 
	$$|pp'|\le \sum_{l=0}^\infty B_l \le \frac{4k+1}{3k}A_0\le \frac{4k+1}{3\sqrt{k}}|F_k(x_0)F_k(p)|.$$
	The conclusion of Theorem \ref{thm-co-Lip-part-dist-coord} now follows.
\end{proof}

\subsection{Proof of Lemma \ref{lem-pushingforward-0}}
In this subsection we prove that a gradient push can be started from $b_{n+1}$ to any point in a very small ball in a definite short time.

Note that if we set $k=n$, $x_0=q$ and $p=b_{n+1}$ in Theorem \ref{thm-co-Lip-part-dist-coord}, then $b_{n+1}$ can be moved to $q$ along gradient curves of $\frac{1}{2}\dist_{a_j}^2$, $\frac{1}{2}\dist_{b_j}^2$ backwards. So we need to reverse the tracing back process defined in the proof of Theorem \ref{thm-co-Lip-part-dist-coord}.

Let $m_n$ be the middle point of $[a_nb_n]$ of a $\delta^2$-maximal $n$-frame $\{[a_jb_j]\}_{j=1}^n$. Let $b_{n+1}$ be a point $\frac{\delta}{100}d_n$-close to $m_n$.
For any $q\in B_{\frac{\delta}{2}d_n}(b_{n+1})$, the \emph{$l$-th round pushing forward point} $O_l$ from $O_0=b_{n+1}$ towards $q$ is defined inductively as follows.

Assume that $O_l$ is well-defined. By tracing back $q_{l+1}^0=q$ to $O_l$ by a single round, we have the $n$-tracing points and tacking broken geodesic $[q_{l+1}^0q_{l+1}^1\cdots q_{l+1}^n]$, where $q_{l+1}^i\in [a_i q_{l+1}^{i-1}]$ or $[b_iq_{l+1}^{i-1}]$ ($i=1,\cdots,n$). Let $\Phi_{l+1}$ be the successive gradient flow defined by 
$$\Phi_{l+1}=\Phi_{1,t_{l+1,1}}\circ\Phi_{2,t_{l+1,2}}\circ\cdots\circ\Phi_{n,t_{l+1,n}}:X\to X,$$
where $\Phi_{i,t_{l+1,i}}$ is the gradient flow of $\frac{1}{2}\dist_{a_i}^2$ or $\frac{1}{2}\dist_{b_i}^2$ which maps $q_{l+1}^i$ to $q_{l+1}^{i-1}$.
We define $O_{l+1}=\Phi_{l+1}(O_l)$.

By (\ref{tracingback-est1}), it is easy to see that the total time satisfies
\begin{equation}\label{round-time}
T_{l+1}=\sum_{i=1}^n{t_{l+1,i}}\le \frac{4(1+\epsilon)}{d_n}\sum_{i=1}^n|f_i(q)-f_i(O_l)|.\tag{4.8}
\end{equation} 

\vspace{2mm}
\begin{proof}[Proof of Lemma \ref{lem-pushingforward-0}]
	~
	
	It suffices to show that the $l$-th round pushing forward point $O_l$ towards $q$ converges to $q$, and the total time admits the bound in Lemma \ref{lem-pushingforward-0}.
	
	Let $A_l=|qO_l|$ and $B_l=\sum_{i=1}^n |f_i(q)-f_i(O_l)|$. Then (\ref{round-time}) can be rewritten as $T_{l+1}\le \frac{4(1+\epsilon)}{d_n}B_l$. 
	
	We first assume that $O_l$ always lies in the cube $$I_{\delta d_n}(m_n)=\{x\in X: |f_i(x)-f_i(m_n)|\le \delta d_n,\; \forall\; 1\le i\le n\}.$$
	Since the Lipschitz constant of distant coordinate function $F_n$ on $I_{\delta d_n}(m_n)$ is almost $1$, 
	\begin{equation}\label{pushingforward-dist-est1}
	|q_{l+1}^nO_l|\le 2\sum_{i=1}^n |f_i(q_{l+1}^n)-f_i(O_l)|,\tag{4.9}
	\end{equation}  
	where by (\ref{tracingback-est2})
	\begin{equation}\label{pushingforward-dist-est2}
	\sum_{i=1}^n |f_i(q_{i+1}^n)-f_i(O_l)|\le \epsilon (n-1) \sum_{i=1}^n |f_i(q)-f_i(O_l)|\le \epsilon (n-1)n |qO_l|.\tag{4.10}
	\end{equation} 
	By $|f_i(q)-f_i(O_{l+1})|\le |qO_{l+1}|$,
	\begin{align*}
	B_{l+1}=\sum_{i=1}^n|f_i(q)-f_i(O_{l+1})| \le n A_{l+1}.  \end{align*}
	
	The concavity of $\frac{1}{2}\dist_x^2$ with $\dist_x<2$ is bounded by $2 \frac{\cosh2}{\sinh2}$. By Theorem \ref{thm-gradient-lip}, (\ref{round-time}) and (\ref{pushingforward-dist-est1})-(\ref{pushingforward-dist-est2}), 
	\begin{align*}
	A_{l+1}=d(q,O_{l+1}) & \le e^{2 \frac{\cosh2}{\sinh2} \cdot T_{l+1}}|q_{l+1}^nO_l|\\
	& \le  e^{8 \frac{\cosh2}{\sinh2} \cdot (1+\epsilon) B_l/d_n} \cdot 2\epsilon (n-1)n |qO_l|\\
	& = e^{8 \frac{\cosh2}{\sinh2} \cdot (1+\epsilon)n A_l/d_n} \cdot 2\epsilon (n-1)n A_l.  
	\end{align*}

	Let us take $\delta(n)>0$ such that for $0<\delta\le\delta(n)$, $A_0/d_n\le \frac{\delta}{2}\le \frac{1}{(1+\epsilon)n}$, and  $\epsilon\le \frac{1}{8n^3 e^{8 \frac{\cosh2}{\sinh2} }}$. Then $(1+\epsilon)n A_0/d_n\le 1$. Moreover, $A_1\le \frac{1}{4n} A_0\le A_0$. By induction, for any $l$, $A_l\le \frac{1}{(4n)^l}A_0$, $B_l\le \frac{n}{(4n)^l} A_0$, and $O_l$ lies in $I_{\delta d_n}(m_n)$. 
	
	Therefore, all estimates above are valid for $0<\delta<\delta(n)$, and $A_l\to 0$ as $l\to \infty$, i.e., $O_l\to q$. Moreover, the total time
	\begin{align*}
	T&=\sum_{i=1}^\infty T_{l} \le 4(1+\epsilon)n\sum_{i=1}^\infty \frac{A_l}{d_n} \\
	&\le 2(1+\epsilon)\delta\frac{4n^2}{4n-1}.
	\end{align*}
\end{proof}

\bibliographystyle{plain}
\bibliography{almost_submetry_references}

\end{document}